\documentclass[10pt]{amsart}
\usepackage{amsthm}
\usepackage{graphicx}
\usepackage{amssymb}
\usepackage{color}
\usepackage{tikz}
\usepackage{amsmath}
\usepackage{hyperref}
\usepackage{float}
\usepackage{enumerate}

\DeclareMathAlphabet\mathbfcal{OMS}{cmsy}{b}{n}

\numberwithin{equation}{section}

\newtheorem{theorem}{Theorem}[section]
\newtheorem{definition}[theorem]{Definition}

\newtheorem{lemma}[theorem]{Lemma}
\newtheorem{conjecture}[theorem]{Conjecture}
\newtheorem{proposition}[theorem]{Proposition}
\newtheorem{corollary}[theorem]{Corollary}
\newtheorem{remark}[theorem]{Remark}
\newtheorem{problem}[theorem]{Problem}

\newcommand{\DD}{{\mathbb D}}
\newcommand{\PP}{{\mathbb P}} \newcommand{\RR}{\mathbb{R}}
\newcommand{\QQ}{\mathbb{Q}} \newcommand{\CC}{{\mathbb C}}
\newcommand{\ZZ}{{\mathbb Z}} \newcommand{\NN}{{\mathbb N}}

\newcommand{\cF}{\mathcal{F}}

\newcommand{\cB}{\mathcal{B}}

\newcommand{\cE}{\mathcal{E}}

\newcommand{\cH}{\mathcal{H}}

\newcommand{\cC}{\mathcal{C}}

\newcommand{\cJ}{\mathcal{J}}

\newcommand{\cO}{\mathcal{O}}
\newcommand{\cP}{\mathcal{P}}

\newcommand{\cR}{\mathcal{R}}

\newcommand{\Ric}{\operatorname{Ric}}
\newcommand{\vol}{\operatorname{vol}}
\newcommand{\ord}{\mathrm{ord}}

\newcommand{\lct}{\mathrm{lct}}

\newcommand{\ddc}{dd^c}

\def\K{K\"ahler } 
\def\KE{K\"ahler--Einstein } \def\KEno{K\"ahler--Einstein}

\def\KRS{K\"ahler--Ricci soliton }
\def\KRSno{K\"ahler--Ricci soliton}

\def\beq{\begin{equation}}
\def\eeq{\end{equation}}
\def\beqno{\begin{equation*}}
\def\eeqno{\end{equation*}}
\def\eaeq{\end{aligned}}
\def\baeq{\begin{aligned}}

\def\bpf{\begin{proof}}
\def\epf{\end{proof}}

\def\ra{\rightarrow}

\def\bclaim{\begin{claim}}
\def\eclaim{\end{claim}}
\def\bdefin{\begin{definition}}
\def\edefin{\end{definition}}
\def\bcor{\begin{corollary}}
\def\ecor{\end{corollary}}
\def\bthm{\begin{theorem}}
\def\ethm{\end{theorem}}
\def\bconj{\begin{conjecture}}
\def\econj{\end{conjecture}}
\def\blem{\begin{lemma}}
\def\elem{\end{lemma}}
\def\blemma{\begin{lemma}}
\def\elemma{\end{lemma}}
\def\bprop{\begin{prop}}
\def\eprop{\end{prop}}
\def\bprob{\begin{problem}}
\def\eprob{\end{problem}}
\def\bremark{\begin{remark}}
\def\eremark{\end{remark}}

\theoremstyle{remark}

\def\lb{\label}
\def\er{\eqref}
\def\noi{\noindent}

\def\meds{\medskip}

\def\maxop{\operatorname{max}}
\def\pis{\pi^\star }

\def\o{\omega}
\def\FS{\operatorname{FS}}
\def\vp{\varphi}

\def\mth{{$m$}}
\def\mthit{$m^{\text{\tiny th}}$}

\title{Basis divisors and balanced metrics\\
}

\usepackage{hyperref}
\hypersetup{colorlinks,linkcolor={blue},citecolor={red}}

\begin{document}

\author[]{Yanir A. Rubinstein}
\address{
Department of Mathematics, University of Maryland, College Park, MD 20742, USA
}
\email{yanir@alum.mit.edu}

\author[]{Gang Tian}
\address{Beijing International Center for Mathematical Research, Peking University, Beijing, 100871, China}
\email{tian@math.pku.edu.cn}

\author[]{Kewei Zhang}
\address{School of Mathematical Sciences, Beijing Normal University, Beijing, 100875, China.}
\email{kwzhang@pku.edu.cn}

\maketitle

\noi{\bf Abstract.} 
Using log canonical thresholds and basis divisors Fujita--Odaka introduced purely algebro-geometric invariants 
$\delta_m$ whose limit in $m$ is now known to characterize uniform K-stability on a Fano variety. As shown by Blum--Jonsson this carries over to a general polarization, and together with work of Berman, Boucksom, and Jonsson, it is now known that 
the limit of these $\delta_m$-invariants characterizes uniform Ding stability. A basic question since Fujita--Odaka's work has been to find an analytic interpretation of these invariants. We show that each $\delta_m$ is the coercivity threshold of a 
quantized Ding functional on the \mthit\ Bergman space and thus characterizes the existence of balanced metrics.
This approach has a number of applications. The most basic one is that it provides an alternative way to compute these
invariants, which is new even for $\PP^n$. Second, it allows us to introduce algebraically defined invariants that characterize the existence of \KRSno{s} (and the more general $g$-solitons of Berman--Witt Nystr\"om), as well as coupled versions thereof. Third, it leads to approximation results involving balanced metrics in the presence of automorphisms that extend some results of Donaldson. 

\tableofcontents

\section{Introduction}
\lb{introsec}

Complex singularity exponents serve as an important bridge between complex geometry
and algebraic geometry. On the one hand, they measure integrability thresholds for
analytic singularities. On the other hand, as conjectured by Cheltsov and shown by Demailly \cite{CS08} (see also Shi \cite{Shi10}), they 
can be interpreted as log canonical thresholds (lct) of divisors. This link has proven
immensely useful in attacking the \KE problem on a Fano manifold $X$, going back to \cite{Tian87}, see, e.g.,
\cite{Tian89,Tia90,CS08}. The terminology for such invariants has become ``$\alpha$-invariants",
and Demailly's theorem states that the $\alpha$-invariant from \cite{Tian87} 
is the limit of a sequence of thresholds lct${}_m$ defined using the linear
system of $m$-anticanonical divisors $|-mK_X|$, which in turn are each equal
to an analytically defined invariant $\alpha_m$ introduced in \cite{Tian89} 
that measures integrability thresholds of \K potentials associated to 
the Kodaira embedding of that linear system.

Recently, Fujita--Odaka \cite{FO18} proposed an approach to the \KE problem by studying 
log canonical thresholds of a particular class of divisors in $|-mK_X|$, called {\it basis
divisors}. This approach yields a sequence of so-called {\it stability thresholds}
(a common terminology here is ``$\delta$-invariants") denoted $\delta_m(-K_X)$
and their limit, denoted $\delta(-K_X)$, detects uniform K-stability by op. cit. and
work of Blum--Jonsson \cite{BJ17}.
However, missing from
this picture is an analytic interpretation of these purely algebraic Fujita--Odaka invariants. Such an interpretation
seems highly desirable, especially given the success in the $\alpha$/lct world. 
In this article we obtain such an analytic interpretation to the Fujita--Odaka approach,
which can also be viewed as an analogue of the Demailly--Tian results in the world of basis divisors.

Our main contribution is to provide an
analytic counterpart of the 
$\delta_m$-invariant, denoted $\delta_m^A$.
Moreover, as it turns out, the analytic approach in the $\delta$-setting yields a number of 
useful applications to canonical metrics that are new, and completely absent from the $\alpha$-setting.
While each $\alpha_m$-invariant does not have a clear geometric application, the analytic
\mthit\ stability threshold $\delta^A_m$ that we introduce here turn out to characterize
balanced metrics. Moreover, they serve as coercivity thresholds for certain quantized Ding functionals.
They are very much computable, in some instances more so than their algebraic counterparts that
up until now were unknown even for $\PP^n$.
Since we show that the two actually coincide, 
$$\delta^A_m=\delta_m,
$$
this proves quite useful in a number of situations. Moreover, via the work of Blum--Jonsson \cite{BJ17} 
this shows that our analytic invariants converge to the Fujita--Odaka $\delta$, and so this
gives via the Yau--Tian--Donaldson framework a new approach to existence and stability. 
Once the connection of our invariants to balanced metrics is proven, one realizes that this
framework is quite flexible, and indeed we show it extends general polarized manifolds $X$ and characterizes
twisted
\KE metrics, \KRSno{s}, coupled \KE metrics, among other canonical metrics. This ties quite neatly with
work of Donaldson,  Berman--Boucksom--Guedj--Zeriahi, Berman--Witt Nystr\"om, Berman--Bouckom--Jonsson and others
\cite{Don01,BBGZ,BWN14,BBJ18}
on relations between balanced metrics, stability, and existence
of canonical metrics. 
We mention that our analysis also extends to the setting of klt currents, 
following Berman--Boucksom--Eyssidieux--Guedj--Zeriahi
and Berman--Boucksom--Jonsson.

\section{Results}

Let $L$ be an ample $\QQ$-line bundle over an $n$-dimensional projective manifold $X$. We assume throughout
\begin{equation}
    \label{eq:m-divisible-mL-ample}
    m\in\NN \text { is sufficiently divisible and } mL\text{ is very ample.}
\end{equation}
Set
$$
d_m:=\dim H^0(X,mL).
$$
The following notion was introduced by Fujita--Odaka.
\begin{definition}
\label{definition:basisdiv}
We say that $D\sim_{\QQ}L$ is a {\rm basis divisor} if for
some $m\in\NN$,
$$
D=\frac{1}{md_m}\sum_{i=0}^{d_m}(s_i),
$$
where
$s_1,...,s_{d_m}$ is a basis of $H^0(X,mL)$, and where
$
(s_i)
$
is the divisor cut out by $s_i$.
We also say that $D$ is the
$m$-basis divisor associated to
the basis $\{s_i\}_{i=1}^{d_m}$.
\end{definition}

Following Fujita--Odaka \cite[Definition 0.2]{FO18}, set 

\beq
\label{def:delta-m-delta}
\delta_m(L):=\inf\Big\{\lct(X,D)\,:\,\ \hbox{\mth}\text{-basis divisor $D$ of } L\Big\},
\eeq
(see \S\ref{lctsubsec} for the definition of $\lct$)
also referred to as the \mth-basis log canonical threshold (blct${}_m$) or 
\mth-stability threshold \cite{CRZ19,BJ17},
and set
\beq
\lb{deltameq}
\delta(L):=\limsup_m\delta_m(L)=\lim_m\delta_m(L),
\eeq
where the last equality is due to Blum--Jonsson \cite[Theorem A]{BJ17}.

The following result of Berman--Boucksom--Jonsson says that 
$\delta=1$ is a threshold for existence of solutions
of the K\"ahler--Einstein type equations
\beq
\lb{thetaKE}
\Ric(\omega)=\omega+\theta,
\eeq 
where $\theta$ is a semipositive smooth $(1,1)$-form in $c_1(X)-c_1(L)$.
Such $\o$ are sometimes called  $\theta$-twisted \KEno, and for
brevity we will simply call them {\it $\theta$-\KEno}.

\begin{theorem}\cite[smooth version of Theorem A]{BBJ18}
\label{thm:BBJ-YTD}
Let $L$ be an ample $\QQ$-line bundle on $X$ and let 
$\theta$ be smooth and semipositive with $[\theta]=c_1(X)-c_1(L)$. 
Then,
\begin{enumerate}[(i)]
    \item if $\delta(L)>1$, there exists a solution
    to \er{thetaKE};
    
    \item if there exists a solution
    to \er{thetaKE} then $\delta(L)\geq1$.
    If such $\o$ is unique then $\delta(L)>1$.
\end{enumerate}
\end{theorem}

\subsection{Algebraic characterization of balanced metrics}

One of the main results of this article is a quantized version of 
Theorem \ref{thm:BBJ-YTD}. 
We show that $\delta_m=1$ is a threshold for the existence of {\it $\theta$-balanced metrics of level $m$} (sometimes referred to 
as twisted balanced metrics);
see \S\ref{sec:alpha-balanced}
for the precise definition.

\begin{theorem}
\label{thm:finite-dim-YTD}
{\rm (Algebraic characterization of balanced metrics)}
Let $L$ be an ample $\QQ$-line bundle on $X$ and let 
$\theta$ be smooth with $[\theta]=c_1(X)-c_1(L)$. 
Then,
\begin{enumerate}[(i)]
    \item if $\delta_m(L)>1$, there exists a
    $\theta$-balanced metric of level $m$;
    
    \item suppose $\theta$ is semipositive. If there exists a  $\theta$-balanced metric of level $m$ then $\delta_m(L)\geq1$.
    If such a metric is unique then $\delta_m(L)>1$.
\end{enumerate}

\end{theorem}

\bremark
{\rm 
Berman--Boucksom--Jonsson prove a more general version of
Theorem \ref{thm:BBJ-YTD} which allows
$\theta$ to be a klt current. 
For conciseness and clarity we assume $\theta$ to be smooth throughtout the main body of this article. 
As we shall see in
Appendix \ref{sec:current}, Theorem \ref{thm:finite-dim-YTD}
also generalizes,
inspired by \cite{BBJ18},
to the setting where $\theta$ is 
merely a klt current;
see Theorem \ref{thm:finite-dim-YTD-current} which can be considered as the quantized version of \cite[Theorem A]{BBJ18}.
}
\eremark

One of the main motivations for studying
balanced metrics is Donaldson's theorem
that shows that balanced
metrics, in the case of finite automorphisms,
approximate constant scalar curvature metrics \cite[Theorem 3]{D05} 
(see \cite{IO20} for an alternative proof in the \KE case); 
we prove in Proposition \ref{prop:alpha-approx-by-0-delta-i}
that a $\theta$-K\"ahler--Einstein metric can be 
approximated by a sequence of suitably twisted balanced metrics 
even when there are possibly non-trivial holomorphic vector fields.

Theorem \ref{thm:finite-dim-YTD} 
relates the differential geometric notion
of a balanced metric with the algebraic geometric
invariant $\delta_m(L)$. To prove this relation we quantize certain 
analytic $\delta$-invariant, studied in detail by one of us \cite{Zha20},
that we now turn to describe; an outline of the 
proof of Theorem \ref{thm:finite-dim-YTD} 
is given at the end of \S\ref{analyticsubsec}.

\subsection{An analytic construction of the $\delta_m$- and $\delta$-invariant}
\lb{analyticsubsec}

The work of Fujita--Odaka, Blum--Jonsson,
and Berman--Boucksom--Jonsson placed at center stage
the algebraically defined $\delta$-invariant by
showing it detects the existence of \KE metrics.
The latter, of course, can be defined purely analytically.
Thus, and in light of the Demailly--Tian results described
in \S\ref{introsec}, a basic question in \KE theory is:

\bprob
\lb{deltanalyticprob}
Can the $\delta_m$- and $\delta$-invariant be computed/defined analytically?
\eprob
\meds

Partial progress on an analytic definition of the $\delta$-invariant
(but not the $\delta_m$-invariant) was obtained by 
Cheltsov--Rubinstein--Zhang \cite[Theorem 5.7]{CRZ19}
and Berman--Boucksom--Jonsson \cite[Theorem C]{BBJ18}. They proved 
that when $\delta$ is at most
$1$ it coincides with the greatest Ricci lower bound $\beta$ going back
to \cite{Tia92,Rub08,Rub09,Sz11}. However, it is precisely the regime where
$\delta>1$ that guarantees existence and in that regime an analytic
insight into $\delta$ is incomplete. What is more, computing $\beta$
itself is challenging and provides further motivation to our study.

We solve Problem \ref{deltanalyticprob} affirmatively. To explain our approach
we introduce some notation first.
Given the polarized pair $(X,L)$ as above, we fix a positively curved smooth Hermitian metric $h$ on $L$. 
The curvature form of $h$, $\omega:=-\ddc\log h$ (where $\ddc:=\frac{\sqrt{-1}}{2\pi}\partial\bar{\partial}$), 
defines a space of \K potentials,
\begin{equation*}
    \cH_\omega:=\Big\{\varphi\in C^\infty(X)\,:\,\omega_\varphi:=\omega+\ddc\varphi>0\Big\}.
\end{equation*}
The Monge--Amp\`ere energy $E(\varphi)$ 
for $\varphi\in\cH_\omega$ is then defined by
\cite[Theorem 2.3]{Mabuchi}
\begin{equation}
\lb{MAener}
    E(\varphi):=E_\omega(
\varphi):=\frac{1}{(n+1)\int_X\omega^n}\sum_{i=0}^n\int_X\varphi\omega^{n-i}\wedge\omega^i_\varphi.
\end{equation}
Following \cite[Definition 3.1]{Zha20}, set
\begin{equation}
\label{deltaA}
\delta^A(L):=\sup\bigg\{\delta>0\,:\,
\sup_{\varphi\in\cH_\omega}
\int_Xe^{-\delta(\varphi-E(\varphi))}\omega^n<\infty
\bigg\}.
\end{equation}
In other words, using an identity of Aubin, $\delta^A(L)$ is the optimal constant in the so-called 
(generalized) Moser--Trudinger inequality \cite{Aub84}.
A conjectural approach to the second part of Problem \ref{deltanalyticprob} is
the following

\bconj
\lb{Zconj}
$\delta^A(L)=\delta(L)$.
\econj

In this article, inspired by Problem \ref{deltanalyticprob}
and Conjecture \ref{Zconj}
we take a quantization approach: First, we introduce a quantized
version of $\delta^A(L)$, that we denote by $\delta^A_m(L)$,
and that is defined analytically; second, we show that 
our invariant $\delta^A_m(L)$ coincides with the Fujita--Odaka
$\delta_m(L)$. This then solves Problem \ref{deltanalyticprob}
and confirms a quantized version of Conjecture \ref{Zconj}. Moreover, our quantization approach also provides a natural way to establish one direction of Conjecture \ref{Zconj}: 
$$\delta^A(L)\leq\delta(L),$$
see Proposition \ref{prop:delta-A<=delta}; see also \cite[Proposition 3.11]{Zha20} for an alternative proof relying the non-Archimedean approach in \cite{BBJ18}. 

To describe our invariants $\delta^A_m(L)$ we introduce some 
additional notation.
Let $\cP_m$ denote the space of all Hermitian inner products on the
complex vector space $H^0(X,mL)$. 
As observed by Donaldson 
\cite[p. 198]{Kob87},\cite{D05} 
a fundamental Bott--Chern type functional on $\cP_m\times\cP_m$ is
\beq
\lb{Em1st}
E_m(H,K):=\frac1{md_m}\log\det K^{-1}H.
\eeq 
In practice it is convenient
to fix some $H$ in the first slot and, in the second slot, 
to pull-back via the isomorphism $\operatorname{FS}:\cP_m\ra\cB_m$ \cite{D05,His17},
\begin{equation*}
	\operatorname{FS}(K):=\frac{1}{m}\log\sum_{i=1}^{d_m}|\sigma_i|^2_{h^m},
\end{equation*}
where $\{\sigma_i\}$ is an(y) orthonormal basis of $K$,
and where $\cB_m$ denotes the image of $\cP_m$ via $\operatorname{FS}$,
also called the \mth\ Bergman space,
\begin{equation*}
    \cB_m:=\bigg\{\varphi=\frac{1}{m}\log\sum_{i=1}^{d_m}|\sigma_i|_{h^m}^2\,:\,\{\sigma_i\}_{i=1}^{d_m}\text{ is a basis of }H^0(X,mL)
    \bigg\}.
\end{equation*}
This yields a functional $E_m\big(H,\FS^{-1}(\,\cdot\,)\big)$, 
that we also denote by
$$
E_m(H,\vp):= 
E_m(H,\FS^{-1}(\vp))=
E_m(H,K),
\quad \hbox{for $\varphi=\FS(K)
\in\cB_m$}.
$$
As shown by Donaldson $E_m$ is 
the natural quantization of $E$ \er{MAener} \cite[\S3]{D05}.
A heuristic
way to understand this is to note that while the gradient of $E$ in the Mabuchi $L^2$ metric on $\cH_\omega$ is the constant vector field $1$,
the gradient of $E_m$ in the symmetric space metric on $\cB_m\cong GL(d_m,\CC)/U(d_m)$ is the left-invariant
field associated to the identity matrix.

\begin{definition}
\label{def:delta-A-m}
Let $L$ be an ample $\QQ$-line bundle. The 
\mthit\ analytic stability threshold is defined by
    \begin{equation*}
   \delta_m^A(L):=\sup\bigg\{\delta>0\,:\,\sup_{\varphi\in\cB_m} \int_Xe^{-\delta(\varphi-E_m(H,\varphi))}\omega^n<\infty\bigg\}.
\end{equation*}
\end{definition}

One should think of $\delta_m^A(L)$ as the optimal constant in a \emph{quantized Moser--Trudinger inequality}.
Note that this definition depends neither on the choice of $H$ 
(due to the cocyclic nature of $E_m$, i.e., $E_m(H,K)+E_m(K,N)=E_m(H,N)$),
nor on the choice of $h$ (and hence of $\o$).

Our next result solves Problem  \ref{deltanalyticprob}
and establishes a quantized version of Conjecture \ref{Zconj}.

\begin{theorem}
\label{thm:delta-A-m=delta-m}
Let $L$ be an ample $\QQ$-line bundle. Then $\delta_m^A(L)=\delta_m(L)$.
\end{theorem}

In fact we will prove a more precise statement, which says that  $\delta_m(L)$ is the coercivity threshold of a quantized Ding type energy on the Bergman space (Proposition \ref{prop:F-coerc=delta-m}).
Combining Theorem \ref{thm:delta-A-m=delta-m} with Blum--Jonsson's theorem \cite[Theorem A]{BJ17} (recall \er{deltameq}), 
we obtain a purely analytic definition of
the $\delta$-invariant.

\begin{corollary}
\label{thm:lim-delta-A-m=delta}
	Let $L$ be an ample $\QQ$-line bundle. Then 
	$\delta(L)=\lim_{m\rightarrow\infty}\delta^A_m(L).$
\end{corollary}

The proof of Theorem \ref{thm:delta-A-m=delta-m} has three key ingredients. First, applying the lower semi-continuity of complex singularity exponent of Demailly--Koll\'ar \cite{DK01}, we show that it suffices to consider basis divisors associated to $H$-orthonormal basis of $H^0(X,mL)$.
Second, it is observed that Donaldson's $E_m$-functional can be related to the \mthit\ expected vanishing order of $L$ along divisors over $X$ (Lemma \ref{lem:E-m=S-m}), which allows us to draw connections between Bergman geodesics and divisorial valuations. The third ingredient is a local computation around the center of a divisorial valuation, 
from which we obtain a uniform integral control along Bergman geodesics
(Lemma \ref{lem:I-lambda>eta}). We combine these ingredients with the valuative description of $\delta_m(L)$ to
conclude Theorem \ref{thm:delta-A-m=delta-m}. 
To obtain the more precise version (Proposition \ref{prop:F-coerc=delta-m}), we need to work a bit harder, by observing that $\sup_X\varphi$ for suitable $\varphi\in\cB_m$ can be related to the \mthit\ pseudo-effective threshold of divisorial valuations via Lemma \ref{lem:sup=max-mu} and \eqref{eq:mu-max=T-m}.

Another corollary of Theorem \ref{thm:delta-A-m=delta-m}
is Theorem \ref{thm:finite-dim-YTD} using
the following variational argument. We consider a certain quantized Ding type functional on the Bergman space $\cB_m$, whose critical points turn out to be balanced metrics in a suitable sense. Using Theorem \ref{thm:delta-A-m=delta-m} we see that $\delta_m(L)$ serves as the coercivity threshold of this functional. Using Berndtsson's convexity \cite{Bern15}, the coercivity principle of Darvas--Rubinstein \cite{DR17}, and Proposition \ref{prop:F-coerc=delta-m}, 
we then conclude Theorem \ref{thm:finite-dim-YTD}.

\subsection{A $\delta$-invariant for K\"ahler--Ricci solitons}

The quantization approach in this article intuitively explains why $\delta$-invariants appear in the K\"ahler--Einstein problem.
It turns out that this viewpoint also naturally leads us to a weighted version of the $\delta_m$-invariant (Definition \ref{def:delta-g-m}) 
for \KRS type metrics, and, moreover, to an associated $\delta$-invariant.
This seems to be the first relation between \KRSno{}s and log canonical
thresholds in algebraic geometry. 
We show that the weighted $\delta_m$-invariant is the stability threshold 
for \emph{quantized soliton metrics} introduced by Berman--Witt Nystr\"om \cite{BWN14}; more precisely, we prove a soliton version of Theorem \ref{thm:finite-dim-YTD} (see Theorem \ref{thm:finite-dim-g-weighted-YTD}). Then building on the recent work of Han--Li \cite{HL20}, an expression of \emph{the greatest Bakry--Emery Ricci lower bound} in terms of the weighted $\delta$-invariant is given (see Proposition \ref{prop:beta-g=delta-g}). This generalizes \cite[Appendix]{CRZ19} and \cite[Theorem C]{BBJ18} to the soliton setting.

Even more generally, 
in
Appendix \ref{sec:coupled-soliton},
we consider coupled soliton type metrics
(cf. Hultgren--Witt Nystr\"om \cite{HNW19} and Delcroix--Hultgren \cite{DH18}).
We introduce a coupled $\delta_m$-invariant (Definition \ref{def:coupled-delta-m}),  and show 
in Theorem \ref{thm:finite-coupled-dim-g-weighted-YTD}
that it
serves as the stability threshold for
the coupled balanced metrics studied by Takahashi \cite{Tak19}. 

\subsection{$\delta$-invariants and automorphisms}
Finally, we state an interesting byproduct of our quantization approach. Let $T_\CC\cong(\CC^*)^r$ be a complex torus acting effectively and holomorphically on $X$. Assume that the action of $T_\CC$ lifts to $L$. 
By Blum--Jonsson and Golota \cite{BJ17,G19}, to compute $\delta(L)$ 
it is enough to only consider $T_\CC$-invariant divisorial valuations. 
However the same fact is not known at level $m$:

\bprob
Can $\delta_m(L)$ be computed using only $T_\CC$-invariant divisorial 
valuations?
\eprob

In fact, even in the simplest case, $X=\PP^n$, it is unclear how to 
compute  $\delta_m$,
the only known case being $n=2$
due to Park--Won.
By its definition as an infimum, it
follows directly that 
$\delta_m(-K_{\PP^n})\le1$  \cite[\S3]{PW18}, 
but even when $n=2$, to show that $\delta_m(-K_{\PP^2})\ge1$ 
is quite tricky and involves detailled
Newton polygon computations \cite[Theorem 3.1]{PW18}.
An alternative analytic approach is furnished by
the following general result.

Let
\beq
\lb{nefsL}
s(L):=\sup\{s\in\RR|-K_X-sL\text{ nef}\,\}
\eeq
denote the nef threshold of $L$. 

\begin{theorem}
\label{thm:delta-T-m}
	Let $L$ be an ample $\QQ$-line bundle 
and assume \er{eq:m-divisible-mL-ample}. Then,
	$$
	\min\{\delta_m^{T_\CC}(L),s(L)\}\le\delta_m(L)\le\delta_m^{T_\CC}(L),
	$$
	where $s(L)$ is defined in \er{nefsL}, and $\delta_m^{T_\CC}(L)$ 
	is defined in \eqref{eq:def-delta-T-m}.
\end{theorem}

In particular, when $L=-K_X$ and $\delta^{T_\CC}_m(-K_X)\leq1$, it follows that $\delta_m(-K_X)=\delta^{T_\CC}_m(-K_X)$.
This result is new already for $X=\PP^n$ and shows that
$\delta_m(-K_{\PP^n})=1$ for all $m$ and $n$. 
Quite more generally,
in Corollary \ref{cor:delta-m-toric-Fano} we 
use Theorem \ref{thm:delta-T-m}
to compute $\delta_m(-K_X)$ for any toric Fano $X$ (cf.
\cite[Section 7]{BJ17} for an algeraic approach that
seemingly only yields the limit in $m$ of these invariants; compare also the very recent work  of Zhuang where an equivariant statement for K-unstable log Fano pairs is obtained using a completely different approach
\cite[Theorem 4.4]{Zhua20}). Our results emphasize some interesting dichotomies between $\alpha$-invariants
and $\delta$-invariants. First, while usually it is easier to compute
$\alpha$-invariants using the algebraic definition, for $\delta$-invariants
it turns out that our analytic definition could be more computable,
at least in some cases. Second, while $\alpha_m$-invariants are 
conjecturally 
constant for sufficiently large $m$, one finds in explicit
examples that this is not at all the case for $\delta_m$-invariants 
(e.g., when $X=Bl_p\PP^2$). 

To prove Theorem \ref{thm:delta-T-m}, the first step is to restrict certain suitably twisted quantized Ding functional onto the $T$-invariant Bergman space $\cB^T_m$, whose coercivity is determined by $\delta^{T_\CC}_m(L)$. Then using the equation of balanced metrics and Berndtsson convexity \cite{Bern15} (this is where $s(L)$ comes into play), we observe that  any critical point of this restricted functional 
in $\cB^T_m$ is necessarily a global minimizer in the entire Bergman space $\cB_m$. Then Theorem \ref{thm:delta-T-m} follows from Theorem \ref{thm:delta-A-m=delta-m}.

\subsection{Organization}
The rest of this article is organized as follows. In Section \ref{sec:analy-delta} we reformulate the definition of $\delta_m(L)$ using analytic language and prove Theorem \ref{thm:delta-A-m=delta-m} and Corollary \ref{thm:lim-delta-A-m=delta}.
The main theme of this article is presented in Section \ref{sec:balanced-metric}, where the $\delta_m$-invariant is shown to fit naturally into a variational picture portraying quantized Ding functionals and balanced metrics, from which we prove Theorem \ref{thm:finite-dim-YTD}. In Section \ref{sec:limit-behavior}, we prove several approximation  results regarding the limit behavior of $\delta_m(L)$ and twisted balanced metrics. In Sections \ref{sec:soliton} we extend the main body of the article to the soliton setting. Then in Section \ref{sec:delta-T}, we prove Theorem \ref{thm:delta-T-m} and give the formula of $\delta_m$-invariants for toric Fano manifolds.
In Appendix \ref{sec:coupled-soliton}, we
extend our main results to the coupled setting. Finally, in Appendix \ref{sec:current}, we extend the analysis to the setting where $\theta$ is a klt current.

\section{Analytic reformulation of $\delta_m(L)$}
\label{sec:analy-delta}

In this section
we prove Theorems \ref{thm:delta-A-m=delta-m} and \ref{thm:lim-delta-A-m=delta}.
To that end, we start by recalling the valuative description of $\delta_m(L)$ \cite{FO18,BJ17}. In \S\ref{lctsubsec} 
we use this description to prove an
upper bound for $\delta_m(L)$, i.e.,  $\delta^A_m(L)\geq\delta_m(L)$. 
The harder lower bound is established
in \S\ref{subsec:destabilizing-ray} 
by constructing an optimal destablizing
Bergman geodesic rays from divisors over $X$.

Let $\pi:Y\rightarrow X$ be a proper birational morphism
and let $F\subset Y$ be a prime divisor $F$ in $Y$. We say
that {\it $F$ is a divisor over $X$}.
Let
\beq
\lb{SmF}
S_m(F):=\frac{1}{md_m}\sum_{j=1}^{\infty}\dim H^0(Y,m\pi^\star L-jF)
\eeq
denote the \emph{\mthit\ expected vanishing order of $L$ along $F$}
(the sum, of course, only runs up to a certain finite $j$ that
will be defined shortly).
Also note that one has
$$
S_m(F)=\sup\big\{\ord_F(D)\,:\,m\text{-basis divisor $D$ of }L\big\},
$$
and this supremum is attained by any $m$-basis divisor $D$ arising from a basis $\{s_i\}$ that is compatible with the filtration
\beq
\lb{filtration}
\{H^0(Y,m\pi^\star L-jF)\}_{j=0}^{\tau_m(\pi^\star L,F)}, 
\text{ where }
\tau_m(\pis L,F):=\max\{x\in\NN \,:\, H^0(Y,m\pi^\star L-xF)\neq0\},
\eeq
i.e., each $H^0(Y,m\pi^\star L-jF)$
is spanned by a subset of the $\{s_i\}_{i=1}^{d_m}$
\cite[Lemma 2.2]{FO18} (see \cite[Lemma 2.7]{CRZ19} for an exposition).
Here $\ord_F(D)$ is the vanishing order of $\pis D$ along $F$.
The log discrepancy of $F$ over $X$ is defined by
\beq\lb{AXF}
A_X(F):=1+\ord_F(K_Y-\pi^\star K_X).
\eeq
Then by \cite{FO18,BBJ18}, 
\begin{equation}
\lb{deltam2eq}
	\delta_m(L)=\inf_{F\, \text{over}\, X}\frac{A_X(F)}{S_m(F)}.
\end{equation}
A well-known fact is that this infimum is attained by some $F$, see Lemma \ref{lem:finite-lct-delta-g-m}
for the proof in a more general setting.
Finally, 
for any effective $\RR$-divisor $D\subset X$, its \emph{log canonical threshold} is defined by
\begin{equation}
\label{lctdef}
    \lct(X,D):=\inf_F\frac{A_X(F)}{\ord_F(D)}.
\end{equation}

\subsection{Reduction to orthonormal basis and an upper bound}
\lb{lctsubsec}
In Proposition \ref{prop:deltam=int<C} we show that to compute $\delta_m$-invariant, it is enough to consider all the orthonormal basis of $H^0(X,mL)$ with respect to 
a fixed Hermitian inner product $H\in\cP_m$. We apply this
to conclude that $\delta_m(L)\le \delta_m^A(L)$ (Corollary \ref{cor:delta-m<=delta-m-A}).

\begin{proposition}
\label{prop:deltam=int<C}
For any  $H\in\cP_m$,
	$$
	\delta_m(L)=\sup\bigg\{\delta>0\,:\,
\sup_{
\substack{
\{s_i\} H\text{-o.n.b.}
}
}
\int_X\frac{\omega^n}
{\prod_{i=1}^{d_m}|s_i|_{h^m}^{\frac{2\delta}{md_m}}}<\infty 
\bigg\}.
	$$
\end{proposition}

\bremark
\lb{deltaHRem}
{\rm 
For  $H\in\cP_m$, 
denote by 
$\delta_m(L;H)$ the right hand side in the statement.
A consequence of Proposition \ref{prop:deltam=int<C} is
that it is independent of $H$, i.e., $\delta_m(L;H)=\delta_m(L;K)$ for
any $H,K\in\cP_m$.
}
\eremark

\begin{proof}
\meds
\noindent
We claim that, in the notation of Remark \ref{deltaHRem}, one has
\begin{equation}
\label{eq:delta-m-omega=int_x<infty}
    \delta_m(L;H)= \sup\bigg\{\delta>0\,:\,\int_X\frac{\omega^n}{\prod_{i=1}^{d_m}|s_i|_{h^m}^{\frac{2\delta}{md_m}}}<\infty,\quad\text{for all $H$-orthonormal bases }\{s_i\}_{i=1}^{d_m}\bigg\}.
\end{equation}
Indeed, denote the RHS of \eqref{eq:delta-m-omega=int_x<infty} by $\tilde{\delta}_m(L;H)$. Then clearly $\delta_m(L;H)\leq\tilde{\delta}_m(L;H)$. If $\delta_m(L;H)<\tilde{\delta}_m(L;H)$, then we can find $\delta\in(\delta_m(L;H),\tilde{\delta}_m(L;H))$ and a sequence of $H$-orthonormal bases $\{s_i^{(j)}\}_{i=1}^{d_m}$ such that
$$
\lim_{j\rightarrow\infty}
\int_X\frac{\omega^n}{\prod_{i=1}^{d_m}|s^{(j)}_i|_{h^m}^{\frac{2\delta}{md_m}}}=\infty.
$$
Up to a subsequence, $\{s_i^{(j)}\}$ converges smoothly to an $H$-orthonormal basis $\{s^{(\infty)}_i\}$. Then by the lower semi-continuity of complex singularity exponents \cite[Theorem 0.2(3)]{DK01},
$$
\lim_{j\rightarrow\infty}\int_X\frac{\omega^n}{\prod_{i=1}^{d_m}|s^{(j)}_i|_{h^m}^{\frac{2\delta}{md_m}}}=\int_X\frac{\omega^n}{\prod_{i=1}^{d_m}|s^{(\infty)}_i|_{h^m}^{\frac{2\delta}{md_m}}}
<\infty,
$$
a contradiction.
This proves \eqref{eq:delta-m-omega=int_x<infty}.

Now for any $F$ over $X$, we consider the filtration \er{filtration}. 
Given  $H\in\cP_m$, observe that one can choose a compatible $H$-orthonormal basis $\{s_i\}$ so 
$S_m(F)=\ord_F(\pi^\star D)$ with $D$ the basis divisor associated to $\{s_i\}_{i=1}^{d_m}$
(recall Definition \ref{definition:basisdiv}).
Namely, 
\beq
\lb{EV2eq}
S_m(F)=\sup\big\{\ord_F(D)\,:\,m\text{-basis divisor $D$ arising from $H$-orthonormal basis}\big\}.
\eeq
Combining \er{lctdef} and \er{EV2eq},
$$
\delta_m(L)=\inf\big\{\lct(X,D)\,:\, m\text{-basis divisor $D$ arising from $H$-orthonormal basis}\big\}.
$$
Thus, by the analytic interpretation of lct \cite[\S 8]{Kol97},
\begin{equation*}
		\delta_m(L)=
\sup\bigg\{\delta>0\,:\,\int_X\frac{\omega^n}{\prod_{i=1}^{d_m}|s_i|_{h^m}^{\frac{2\delta}{md_m}}}<\infty,\quad\text{for all $H$-orthonormal bases }\{s_i\}_{i=1}^{d_m}\bigg\}.
\end{equation*}
Combining this with \eqref{eq:delta-m-omega=int_x<infty} concludes
the proof.
\end{proof}

\begin{corollary}
\label{cor:delta-m<=delta-m-A}
$
	\delta_m(L)\le\delta^A_m(L).
$

\end{corollary}

	\begin{proof}
We first reformulate
Definition \ref{def:delta-A-m}.
Fix a reference 
Hermitian inner product $H\in\cP_m$. 
Then $E_m$ \er{Em1st} can be expressed by
\beq\lb{Em2nd}
E_m(H,\varphi)=\frac{1}{md_m}\log\det\big[H(\sigma_i,\sigma_j)\big]_{i,j=1}^{d_m},
\eeq
for any $\varphi=\FS(K)=
\frac{1}{m}\log\sum_{i=1}^{d_m}|\sigma_i|_{h^m}^2\in\cB_m$, 
where $\{\sigma_i\}$ is $K$-orthonormal.
By linear algebra, for any basis $\{\sigma_i\}$ of $H^0(X,mL)$, after a unitary transformation, one may diagonalize it so that
$$
\sigma_i=\mu_i^{1/2} s_i
$$
for some $H$-orthonormal basis $\{s_i\}$, with $\mu_i>0$. Using such convention, one can also write
\begin{equation}
\label{eq:E-m=prod-mu-i}
    E_m(H,\varphi)=\frac{1}{md_m}\log\prod_{i=1}^{d_m}\mu_i.
\end{equation}
Note that a different choice of $H$ will only shift $E_m$ by a constant.
Thus, Definition \ref{def:delta-A-m} becomes
	\beq
\lb{deltaA2eq}
\delta^A_m(L)=\sup\bigg\{\delta>0\,:\,
\sup_{
\substack{
\{s_i\} H\text{-o.n.b.}\\
\mu_i>0}
}
\int_X\frac{\prod_{i=1}^{d_m}\mu_i^{\frac{\delta}{md_m}}}{\bigg(\sum_{i=1}^{d_m}\mu_i|s_i|^2_{h^m}\bigg)^{\frac{\delta}{m}}}\omega^n<\infty
\bigg\}.
   \eeq
By the arithmetic mean--geometric mean inequality,
\begin{equation*}
\sum_{i=1}^{d_m}\mu_i |s_i|^2_{h^m}
\ge
d_m
\Big(\prod_{i=1}^{d_m}\mu_i\Big)^{\frac{1}{d_m}}
\Big(\prod_{i=1}^{d_m}|s_i|^2_{h^m}\Big)^{\frac{1}{d_m}},
\end{equation*}
thus for any $H$-orthonormal basis $\{s_i\}$ and parameters $\mu_i>0$, 
	\begin{equation*}
			\int_X\frac{\prod_{i=1}^{d_m}\mu_i^{\frac{\delta}{md_m}}}{(\sum_{i=1}^{d_m}\mu_i |s_i|^2_{h^m})^{\frac{\delta}{m}}}\omega^n
			\le\bigg(\frac{1}{d_m}\bigg)^{\frac{\delta}{m}}\cdot\int_X\frac{\omega^n}{\prod_{i=1}^{d_m}|s_i|^{\frac{2\delta}{md_m}}_{h^m}},
		\end{equation*}
and the statement follows from Proposition \ref{prop:deltam=int<C}.
\end{proof}

\subsection{Optimal destabilization and a lower bound}
\label{subsec:destabilizing-ray}

We now turn to proving the harder inequality,
\begin{equation}\lb{lowerbnd}
	\delta_m(L)\ge\delta^A_m(L).
\end{equation}
Our strategy is as follows. Fix a prime divisor $F$ over $X$. We find
 an $H$-orthogonal basis $\{s_i\}$ of $H^0(X,mL)$ 
such that the integral
\begin{equation}
\label{eq:int-def-delta-A-m}
	\int_X\frac{\prod_{i=1}^{d_m}\mu_i^{\frac{\delta}{md_m}}}{\bigg(\sum_{i=1}^{d_m}\mu_i|s_i|^2_{h^m}\bigg)^{\frac{\delta}{m}}}\omega^n
\end{equation}
has no uniform upper bound for an appropriate choice 
of positive numbers $\{\mu_i\}$, whenever $\delta$ satisfies
$
\delta>\frac{A_X(F)}{S_m(F)}.
$
This  implies
$
\delta^A_m(L)\le\inf_{F}\frac{A_X(F)}{S_m(F)}=\delta_m(L),
$
i.e., \er{lowerbnd}, which, when combined with Corollary
\ref{cor:delta-m<=delta-m-A} will conclude the proof of 
 Theorem \ref{thm:delta-A-m=delta-m}.

\bdefin
\lb{FsiBergGeod}
Let $F$ be a prime divisor over $X$
and let $\{s_i\}$ be an $H$-orthonormal basis  
of $H^0(X,mL)$ compatible with the filtration \eqref{filtration}.
The Bergman
geodesic ray associated to $(F,\{s_i\})$ is
\beq
\lb{vpF}
\varphi_F(t):=\frac{1}{m}\log\sum_{i=1}^{d_m}e^{t\,\ord_F(s_i)}|s_i|^2_{h^m}\in\cB_m,
\quad t\in\RR.
\eeq

\edefin
A simple, but key, observation is that
the \mthit\ expected vanishing can
be viewed as the slope of the
Monge--Amp\`ere energy.

\blem
\label{lem:E-m=S-m}
Let $\vp(t)$ be defined by \er{vpF}. Then, 
$E_m(H,\varphi_F(t))=tS_m(F).
$
\elem

\bpf
By \er{eq:E-m=prod-mu-i}, 
$
E_m(H,\varphi_F(t))=
\frac{t}{md_m}\sum_{i=1}^{d_m}\ord_F(s_i). 
$
For any basis $\{s_i\}$ compatible with the filtration,
$$
\baeq
\sum_{i=1}^{d_m}\ord_F(s_i) 
&=\sum_{j=0}^{\infty}j\Big[\dim H^0(Y,m\pis L-jF)-\dim H^0(Y,m\pis L-(j+1)F)\Big]
\cr
&=\sum_{j=1}^{\infty}\dim H^0(Y,m\pis L-jF),
\eaeq
$$
where in the last line we used that $H^0(Y,m\pis L-jF)$
vanishes for large enough $j$ (recall
\er{filtration}). Thus,
by \er{SmF} the proof is complete.
\epf

Fix $F\subset Y$ over $X$ and let $\{s_i\}$ be a basis
as in the proof of Lemma \ref{lem:E-m=S-m}. 
We evaluate \eqref{eq:int-def-delta-A-m} along
the Bergman geodesic $\vp_F(t)$ of Definition \ref{FsiBergGeod},
i.e., put $\mu_i(t)=e^{t\,\ord_F(s_i)}$,
and use Lemma 
\ref{lem:E-m=S-m}, 
\begin{equation}
\lb{deltaEVAX}
	\begin{aligned}
		\eqref{eq:int-def-delta-A-m}(t)&=\int_X
\frac{e^{t\delta S_m(F)}}{\bigg(\sum_{i=1}^{d_m}e^{t\,\ord_F(s_i)}|s_i|^2_{h^m}\bigg)^{\frac{\delta}{m}}}\omega^n\\
		&=e^{t(\delta S_m(F)-A_X(F))}\int_X\frac{e^{tA_X(F)}}{\bigg(\sum_{i=1}^{d_m}e^{t\,\ord_F(s_i)}|s_i|^2_{h^m}\bigg)^{\frac{\delta}{m}}}\omega^n.
			\end{aligned}\\
\end{equation}
Now the key estimate is the following.
\begin{lemma}
\label{lem:I-lambda>eta}
	There exists $C>0$ such that
	$$
	\int_X\frac{e^{tA_X(F)}}{\bigg(\sum_{i=1}^{d_m}e^{t\,\ord_F(s_i)}|s_i|^2_{h^m}\bigg)^{\frac{\delta}{m}}}\omega^n> C>0, \hbox{\ for all $t\ge0$}.
	$$
\end{lemma}

\begin{proof}
	Let $Z:=\pi(F)$ denote the center of the divisorial valuation $\ord_F$ on $X$. We will show that the desired estimate follows from a local calculation around $Z$. More percisely, Let $\Omega$ be a tubular neighborhood around $Z$. It suffices to find some $\eta>0$ such that
	$$
	\int_\Omega\frac{e^{tA_X(F)}}{\bigg(\sum_{i=1}^{d_m}e^{t\,\ord_F(s_i)}|s_i|^2_{h^m}\bigg)^{\frac{\delta}{m}}}\omega^n\ge \eta>0
	$$
	for any $t\ge0$.
	We will achieve this estimate by pulling back everything to $Y$. Note,
	\begin{equation}
	\label{eq:KY=KX+F+D}
	    K_Y=\pis K_X+(A_X(F)-1)F+D,
	\end{equation}
	where $D$ is some divisor whose support does not contain $F$.
Then we choose a small enough coordinate chart 
	$$\bigg(U,(z_1,\cdots,z_n)\bigg)\subseteq Y,$$
	 centered around some smooth point of $F$ with the following properties:
	\begin{enumerate}
		\item $U$ is away from all the other exceptional divisors of $\pi$ (i.e., $U\cap\operatorname{Supp}(D)=\emptyset$);
		\item  Over $U$, one has
	$
	F=\{z_1=0\};
	$
	\item $U$ contains the polydisk
	$
	\DD:=\bigg\{(z_1,...,z_n)\,:\,|z_i|\le 1,\ \forall i\bigg\};
	$
	\item $\pis(mL)$ is trivialized over $\DD$, so that each $\pis s_i$ can be represented as
	$
	\pis s_i=z_1^{\ord_F(s_i)}g_i(z),
	$
	where $g_i(z)$ is some holomorphic function on $\DD$.
	\item In the above trivialization, there exists some constant $C>0$ such that
	$
	h^m<C,\ |g_i|^2<C, \;\forall i.
	$
	\end{enumerate}
Using \eqref{eq:KY=KX+F+D} and (2),
the volume form $\pis\omega^n$ can be replaced (up to some bounded factor) by
$$
|z_1|^{2A_X(F)-2}(\sqrt{-1})^ndz_1\wedge \overline{dz_1}\wedge\cdots 
\wedge dz_n\wedge \overline{dz_n},
$$
since we are working away from $D$.

Therefore, to finish the proof of Lemma
\ref{lem:I-lambda>eta}, it suffices to find some constant $c>0$ such that for any $t\ge0$,
$$
\int_\DD\frac{e^{tA_X(F)}|z_1|^{2A_X(F)-2}}{\bigg(\sum_{i=1}^{d_m}e^{t\,\ord_F(s_i)}|z_1|^{2\ord_F(s_i)}|g_i|^2h^m\bigg)^{\frac{\delta}{m}}}(\sqrt{-1})^{n}dz_1\wedge 
\overline{dz_1}\wedge\cdots \wedge dz_n\wedge \overline{dz_n}\ge c>0.
$$
Using condition (5) above, it suffices to bound 
$$
\begin{aligned}
J(t)&:=\sqrt{-1}\int_{|z_1|\le1}\frac{e^{tA_X(F)}|z_1|^{2A_X(F)-2}}{\bigg(\sum_{i=1}^{d_m}|e^{t/2}z_1|^{2\ord_F(s_i)}\bigg)^{\frac{\delta}{m}}}dz_1\wedge d\bar{z_1}
\cr&=\sqrt{-1}\int_{|w|\le e^{t/2}}\frac{|w|^{2(A_X(F)-1)}}{\bigg(\sum_{i=1}^{d_m}|w|^{2\ord_F(s_i)}\bigg)^{\frac{\delta}{m}}}dw\wedge d\bar{w}\\
		&\ge\sqrt{-1}\int_{|w|\le1}\frac{|w|^{2(A_X(F)-1)}}{\bigg(\sum_{i=1}^{d_m}|w|^{2\ord_F(s_i)}\bigg)^{\frac{\delta}{m}}}dw\wedge d\bar{w},\\
	\end{aligned}
$$
where in the last inequality we used $t\ge0$.
This last integral is some positive quantity depending only on $\delta$, $m$, $A_X(F)$ and $\{\ord_F(s_i)\}_{1\le i\le d_m}$.
	This completes the proof
	of Lemma
\ref{lem:I-lambda>eta}.
\end{proof}

\bcor
\lb{lowerbndcor}
$
	\delta_m(L)\ge\delta^A_m(L).
$
\ecor 

\bpf
By \er{deltaEVAX} and
Lemma \ref{lem:I-lambda>eta},
$\lim_{t\rightarrow\infty}\eqref{eq:int-def-delta-A-m}(t)=\infty$
if $\delta>\frac{A_X(F)}{S_m(F)}$. Thus, by \er{deltaA2eq},
$
\delta_m^A(L)\le\frac{A_X(F)}{S_m(F)}.
$
Taking the infimum over all $F$ and using \er{deltam2eq} we conclude.
\epf

\begin{proof}[Proof of Theorem \ref{thm:delta-A-m=delta-m}]
This follows from Corollaries 
\ref{cor:delta-m<=delta-m-A} and \ref{lowerbndcor}.
\end{proof}

\begin{remark}
{\rm 	The argument above shows that any divisor $F$ over $X$ naturally induces a basis $\{s_i\}$ with zero locus concentrating around the center $\pi(F)$, which destabilizes $X$ along the Bergman geodesic ray
$\varphi_F(t)$.
}
\end{remark}

\subsection{Local $\delta$-invariant}
Note that the previous argument also applies to the local setting. 
More precisely, let $\Omega\subseteq\CC^n$ be a domain in $\CC^n$. We consider the Bergman space
$$
\cB(\Omega):=\cO(\Omega)\cap L^2(\Omega).
$$
Namely $\cB$ is the linear space consisting of all the square integrable holomorphic functions on $\Omega$. Let
$$
V\subseteq\cB(\Omega)
$$
be a finite dimensional subspace. Then for any point $x\in\Omega$, one can define a local $\delta$-invariant by puting
\begin{equation*}
	\delta_x(V):=\sup\bigg\{\delta>0\,:\,\text{for $\forall$ basis }\{f_i\}\text{ of }V,\ \exists \text{ nbhd. $x\in U$ s.t.}\int_U\frac{d\mu}{\prod_{i=1}^{\dim V}|f_i|^{\frac{2\delta}{\dim V}}}<+\infty\ \bigg\}.
\end{equation*}
This local invariant also appears in a recent work of Xu--Zhuang \cite{XZ20}.

There is a natural $L^2$ inner product on $V$, given by
$
\langle f,g\rangle:=\int_\Omega f\bar{g}d\mu,\ \forall f,g\in V.
$
So it is convenient to consider orthonormal basis of this inner product.
We have the following local version of Proposition \ref{prop:deltam=int<C}.
\begin{proposition}
One has
	$$
	\delta_x(V)=\sup\bigg\{\delta>0\,:\,\exists\text{ nbhd. }x\in U_\delta,\ C_\delta>0\text{ s.t.}\int_{U_\delta}\frac{d\mu}{\prod_{i=1}^{\dim V}|f_i|^{\frac{2\delta}{\dim V}}}<C_\delta,\forall\text{ orthonormal }\{f_i\}\bigg\}
	$$
\end{proposition}
We also have the following local version of Theorem \ref{thm:delta-A-m=delta-m}.
\begin{proposition}
	One has
	$$
	\delta_x(V)=\sup\Bigg\{\delta>0\,:\,
	\begin{aligned}
		&\exists\text{ nbhd. } 
		U_{\delta}\ni x,\ C_\delta>0\text{ such that\ } \int_{U_\delta}\frac{\prod_{i=1}^{\dim V}\mu_i^{\frac{\delta}{\dim V}}d\mu}{(\sum_{i=1}^{\dim V}\mu_i|f_i|^2)^\delta}<C_\delta
		\cr
		&\text{for\ }\forall\text{ orthonormal basis }\{f_i\}\text{ of }V \text{and any positive parameters }\{\mu_i\}
		\\
	\end{aligned}
	\Bigg\}
	$$
\end{proposition}

\section{Quantized Ding energy and balanced metrics}
\label{sec:balanced-metric}
Fix $H$ be some reference Hermitian inner product on $H^0(X,mL)$.
In this section we show that $\delta_m(L)$ serves as the coercivity threshold for certain quantized Ding functional, whose critical points correspond to balanced metrics. This then allows us to conclude Theorem \ref{thm:finite-dim-YTD}.

\subsection{Variational characterization of balanced metrics}
\label{sec:quantized-Ding-balanced}

\begin{definition}
\label{def:F-f-delta-m}
	Let $f\in C^\infty(X,\RR)$ and $\delta>0$. Define
	$F_m^{f,\delta}:\cB_m\rightarrow\RR$ by
$$
F_m^{f,\delta}(\varphi):=-\frac{1}{\delta}\log\frac{1}{V}\int_Xe^{f-\delta\varphi}\omega^n-E_m(H,\varphi),\quad \varphi\in\cB_m.
$$
	A critical point $\varphi$ of $F_m^{f,\delta}(\cdot)$ is called $(f,\delta)$-balanced.

\end{definition}

Similar functionals also appeared in, e.g., 
\cite{BBGZ,Ber13,IO20}. As we will see in Proposition \ref{prop:F-coerc=delta-m}, this functional is naturally related to the $\delta_m$-invariant.

\blem
Let $H_0\in\cP_m$. Then
$\varphi=\FS(H_0)\in\cB_m$
is $(f,\delta)$-balanced if and only if 
\begin{equation*}
H_0=\operatorname{Hilb}^{f,\delta}(\varphi):=
\frac{d_m}{\int_Xe^{f-\delta\varphi}\omega^n}
\int_Xh_\varphi^m(\,\cdot\,,\,\cdot\,)e^{f-\delta\varphi}\omega^n,
\end{equation*}
where 
$h_\varphi:=he^{-\varphi}$.
Namely, $H_0=\operatorname{Hilb^{f,\delta}}\circ\operatorname{FS}(H_0)$.
\elem
\bpf
Let $\{\sigma_i\}$ be an $H_0$-orthonormal basis. By definition,
$
\varphi=\FS(H_0)=\frac{1}{m}\log\sum_{i=1}^{d_m}|\sigma_i|^2_{h^m}\in\cB_m.
$
A computation shows that if $\vp$ is a critical point of 
$F_m^{f,\delta}(\cdot)$ then
\begin{equation}
\label{eq:balanced-equation}
	\frac{d_m}{\int_Xe^{f-\delta\varphi}\omega^n}
	\int_Xh^me^{-m\varphi}(\sigma_i,\sigma_j)e^{f-\delta\varphi}\omega^n=\delta_{ij},\ \forall\ 1\le i,j\le d_m,
\end{equation}
as desired.
\epf


\begin{remark}
{\rm One should think of $F_m^{f,\delta}$ as the quantization of the following Ding energy:
	\beq\lb{Dingfdelta}
	F^{f,\delta}(\varphi):=-\frac{1}{\delta}\log\frac{1}{V}\int_Xe^{f-\delta\varphi}\omega^n-E(\varphi),\ \varphi\in\cH_\omega.
	\eeq
Any critical point $\varphi$ of $F^{f,\delta}$ must satisfy (after a suitable normalization)
$
(\omega+\ddc\varphi)^n=e^{f-\delta\varphi}\omega^n.
$
This means that $\omega_\varphi:=\omega+\ddc\varphi$ solves the twisted \KE equation,
$
\Ric(\omega)_\varphi=\delta\omega_\varphi+\theta,
$
where $\theta:=\Ric(\omega)-\delta\omega-\ddc f$
is a smooth form representing
$c_1(X)-\delta c_1(L)$ and
 determined by $\omega$, $f$ and $\delta$.
}\end{remark}

Now we introduce the notion of coercivity for functionals on $\cB_m$. 

\begin{definition}
\label{def:coerc-F}
	We say $F_m^{f,\delta}(\cdot)$ is coercive on $\cB_m$ if there exists $\varepsilon,C>0$ such that for any $\varphi\in\cB_m$,
	$$
	F^{f,\delta}_m(\varphi)\ge\varepsilon (\sup\varphi-E_m(H,\varphi))-C.
	$$
\end{definition}

\begin{remark} {\rm
\label{anyfRem}
It follows from Definition \ref{def:coerc-F} that if $F_m^{f,\delta}$ is 
coercive then $F_m^{f^\prime,\delta}$ is coercive for any 
other $f^\prime\in C^\infty(X,\RR)$. Also, coercivity does not 
depend on the choice of the reference Hermitian product $H$. 
} 
\end{remark}
The next existence result is standard.

\begin{proposition}
\label{prop:F-m-coercive=>balanced}
	If $F^{f,\delta}_m$ is coercive on $\cB_m$ 
	there exists a $(f,\delta)$-balanced metric that minimizes it.
\end{proposition}

For its proof, we prepare two lemmas.

	\begin{lemma}
\label{lem:mu-max=sup-E}
There exists $C=C(H)$ such that for any $\varphi=\frac{1}{m}\log\sum_{i=1}^{d_m}\mu_i|s_i|^2_{h^m}\in\cB_m$, where $\{s_i\}$ is $H$-orthonormal and $\mu_i>0$,
	$$
	\frac{1}{md_m}\prod_{i=1}^{d_m}\frac{\max_i\mu_i}{\mu_i}-C
	\le
	\sup\varphi-E_m(H,\varphi)
	\le\frac{1}{md_m}\prod_{i=1}^{d_m}\frac{\max_i\mu_i}{\mu_i}+C.
	$$
	\begin{proof}
		We might as well assume that $H$ is the $L^2$ inner product
		$
		\int_Xh^m(\,\cdot\,,\,\cdot\,)\,\omega^n.
		$
		Then the assertion follows from Lemma \ref{lem:sup=max-mu}.
	\end{proof}
\end{lemma}

The next lemma is implicit in 
computations in \cite{Tia90}.

\begin{lemma} 
\label{lem:sup=max-mu}
There exists $\varepsilon_m\rightarrow0$ such that
	$
	\big|\sup\varphi-\frac{1}{m}\log\max_i\mu_i\big|\leq\varepsilon_m,
	$
	for any $\varphi=\frac{1}{m}\log\sum_{i=1}^{d_m}\mu_i|s_i|^2_{h^m}\in\cB_m$, where $\{s_i\}$ is $H_m$-orthonormal for
	$
    H_m:=\int_Xh^m(\,\cdot\,,\,\cdot\,)\,\omega^n\in\cP_m,
    $
	and $\mu_i>0$.
\end{lemma}

\begin{proof}
Let $s\in H^0(X,mL)$ be any holomorphic section with $H_m(s,s)=1$. Then it is clear that
$$
\sup|s|^2_{h^m}\geq\frac{1}{V}.
$$
On the other hand 
$
\Delta|s|^2_{h^m}=|\nabla s|^2_{h^m}-nm|s|^2_{h^m}
$
\cite[(5.1)]{Tia90} (see also \cite[Lemma 9]{Don01}).
Then standard Moser iteration yields a constant $C=C(X,\omega)>0$ such that
    $$
    \sup|s|^2_{h^m}\leq Cm^n.
    $$
    Thus,
    $
    \sup\big|\varphi-\frac{1}{m}\log\max_i\mu_i\big|\leq C^\prime\frac{\log m}{m},
    $
   as desired.
\end{proof}

\begin{proof}[Proof of Proposition \ref{prop:F-m-coercive=>balanced}]
	By coercivity, one can find a minimizing sequence $\{\varphi_j\}_{j\in\NN}$ such that
	$$
	\sup_j(\sup\varphi_j-E_m(H,\varphi_j))<\infty,\text{ and }\lim_jF_m^{f,\delta}(\varphi_j)=\inf_{\cB_m}F_m^{f,\delta}. 
	$$
Let $\{\mu_i(j)\}$ be 
	the eigenvalues of $\FS^{-1}(\varphi_j)$
	with respect to $H$. 
	By Lemma \ref{lem:mu-max=sup-E}, 
	$0<C^{-1}<\mu_i(j)/\max_i\mu_i(j)<C$ for some
	uniform $C>0$. By normalizing the $\vp_j$ so that $\sup\vp_j=0$, Lemma
	\ref{lem:sup=max-mu} gives that $\max_i\mu_i(j)$ is also uniformly bounded.
	By compactness, established below in Lemma \ref{lem:d-1-cptness} and \er{d1mmu}, we can then extract a limit $\varphi_{\infty}\in\cB_m$ such that  up to a subsequence,
	$$
	\varphi_i\xrightarrow{C^\infty}\varphi_\infty.
	$$
	Then $\varphi_\infty$ is a minimizer of $F^{f,\delta}_m$ on $\cB_m$, which hence is $(f,\delta)$-balanced concluding the proof
of Proposition \ref{prop:F-m-coercive=>balanced}.
\end{proof}

\subsection{Algebraic characterization of coercivity}
The main result of this part is the following algebraic characterization of coercivity.
\begin{proposition}
\label{prop:F-coerc=delta-m}
	$F^{f,\delta}_m$ is coercive on $\cB_m$ if and only if 
$\delta\in(0,\delta_m(L))$.
	In particular, if $\delta\in(0,\delta_m(L))$, there exists an $(f,\delta)$-balanced metric in $\cB_m$ for any $f\in C^\infty(X)$.
\end{proposition}

\begin{proof}
The following invariant goes back to \cite[\S6]{Tian89}, 
\begin{equation*}
	\alpha_m(L):=\sup\bigg\{
	\alpha>0\,:\,\sup_{\varphi\in\cB_m}\int_Xe^{-\alpha(\varphi-\sup\varphi)}\omega^n<\infty\bigg\}.
\end{equation*}

By Remark \ref{anyfRem}, it suffices to consider the case $f=0$. We first show that  for any $\delta\in(0,\delta_m(L))$, $\gamma\in(\delta,\delta_m(L))$ and $\alpha\in(0,\min\{\delta,\alpha_m(L
	)\})$, there exists $C>0$ such that
	$$
	-\frac{1}{\delta}\log\frac{1}{V}\int_Xe^{-\delta\varphi}\omega^n-E_m(H,\varphi)\ge \frac{\alpha(\gamma-\delta)}{\delta(\gamma-\alpha)}(\sup\varphi-E_m(H,\varphi))-C,\ \forall\varphi\in\cB_m.
	$$
	This follows easily from H\"older inequality. Indeed,
	\begin{equation*}
		\begin{aligned}
			\frac{1}{\delta}\log\frac{1}{V}\int_Xe^{-\delta\varphi}\omega^n&=\frac{1}{\delta}\log\frac{1}{V}\int_Xe^{-\alpha\frac{\gamma-\delta}{\gamma-\alpha}\varphi-\gamma\frac{\delta-\alpha}{\gamma-\alpha}\varphi}\omega^n\\
			&\le\frac{\gamma-\delta}{\delta(\gamma-\alpha)}\log\frac{1}{V}\int_Xe^{-\alpha\varphi}\omega^n+\frac{\delta-\alpha}{\delta(\gamma-\alpha)}\log\frac{1}{V}\int_Xe^{-\gamma\varphi}\omega^n\\
			&=\frac{\gamma-\delta}{\delta(\gamma-\alpha)}\log\frac{1}{V}\int_Xe^{-\alpha(\varphi-\sup\varphi)}\omega^n+\frac{\delta-\alpha}{\delta(\gamma-\alpha)}\log\frac{1}{V}\int_Xe^{-\gamma(\varphi-E_m(H,\varphi))}\omega^n\\
			&\ \ \ -\frac{\alpha(\gamma-\delta)}{\delta(\gamma-\alpha)}\sup\varphi-\frac{\gamma(\delta-\alpha)}{\delta(\gamma-\alpha)}E_m(H,\varphi)\\
			&\le C-\frac{\alpha(\gamma-\delta)}{\delta(\gamma-\alpha)}\sup\varphi-\frac{\gamma(\delta-\alpha)}{\delta(\gamma-\alpha)}E_m(H,\varphi)\\
			&=C-\frac{\alpha(\gamma-\delta)}{\delta(\gamma-\alpha)}(\sup\varphi-E_m(H,\varphi))-E_m(H,\varphi).
		\end{aligned}
	\end{equation*}
	From this we obtain that  for any $\delta\in(0,\delta_m(L))$, $F_m^{f,\delta}(\cdot)$ is coercive.
	
	Now conversely, suppose that  for some $\delta>0$, there exist $\varepsilon>0$ and $C>0$ such that
	\begin{equation}
	\label{eq:F-f-delta>e-sup-E-C}
		-\frac{1}{\delta}\log\int_Xe^{-\delta\varphi}\omega^n-E_m(H,\varphi)\ge\varepsilon(\sup\varphi-E_m(H,\varphi))-C,\ \forall\varphi\in\cB_m.
	\end{equation}
	Our goal is to show that $\delta_m(L)>\delta$. The argument will be similar to the one for Theorem \ref{thm:delta-A-m=delta-m}. More precisely, for any $\varphi\in\cB_m$, we write
	$\varphi=\frac{1}{m}\log\sum_{i=1}^{d_m}\mu_i|s_i|^2_{h^m}$
	for some $H$-orthonormal basis $\{s_i\}$ with $\mu_i>0$.
	Put
	\begin{equation*}
		\mu_{\maxop}:=\max_i\{\mu_i\}
	\end{equation*}
	Then by Lemma \ref{lem:mu-max=sup-E}, the coercivity assumption \eqref{eq:F-f-delta>e-sup-E-C} is equivalent to
	$$
	-\frac{1}{\delta}\log\int_Xe^{-\delta\varphi}\omega^n-\frac{1}{md_m}\log\prod_{i=1}^{d_m}\mu_i\ge\frac{\varepsilon}{md_m}\log\prod_{i=1}^{d_m}\frac{\mu_{\maxop}}{\mu_i}-C^\prime
	$$
	for some $\varepsilon>0$ and $C^\prime>0$. Rearranging, we derive that
	\begin{equation}
	\label{eq:int-delta-epsilon-prime}
		\int_X\frac{\bigg(\prod_{i=1}^{d_m}\mu_i^{\frac{\delta(1-\varepsilon)}{md_m}}\bigg)\mu_{\maxop}^{\frac{\delta\varepsilon}{m}}}{\bigg(\sum_{i=1}^{d_m}\mu_i|s_i|^2_{h^m}\bigg)^{\frac{\delta}{m}}}\omega^n\le e^{\delta C^\prime}
	\end{equation}
	for any $H$-orthonormal basis $\{s_i\}$ and parameters $\mu_i>0$. Now for any prime divisor $F$ over $X$, as in the proof of Theorem \ref{thm:delta-A-m=delta-m}, we let $\{s_i\}$ be compatible with the filtration on $H^0(X,mL)$ induced by $\ord_F$. Let $t\ge0$ be a parameter and put
	$$
	\mu_i:=e^{t\,\ord_F(s_i)}.
	$$
	Moreover, consider the \mthit\ pseudo-effective threshold
	$$
	\tau_m(F)=\max_{s\in H^0(X,mL)}\ord_F(s).
	$$
	Then
	\begin{equation}
	\label{eq:mu-max=T-m}
	    \mu_{\maxop}=e^{t\,\tau_m(F)}.
	\end{equation}
	Plugging all these into \eqref{eq:int-delta-epsilon-prime}, we find that
	$$
	e^{t\big(\delta(1-\varepsilon)S_m(F)+\delta\varepsilon \tau_m(F)/m-A_X(F)\big)}\int_X\frac{e^{t\,A_X(F)}\omega^n}{\bigg(\sum_{i=1}^{d_m}e^{t\,\ord_F(s_i)}|s_i|_{h^m}^2\bigg)^{
	\frac{\delta}{m}}}\le e^{\delta C^\prime},
	$$
	for any $t\ge0$. Then by Lemma \ref{lem:I-lambda>eta}, we must have
	$$
	\delta(1-\varepsilon)S_m(F)+\delta\varepsilon \tau_m(F)/m-A_X(F)\le0,\ \forall F.
	$$
Now further assume that $S_m(F)\neq0$.
	Then we get
	$$
	\frac{A_X(F)}{S_m(F)}\ge\delta(1-\varepsilon)+\delta\varepsilon\frac{\tau_m(F)}{mS_m(F)}.
	$$
	To finish the proof, it is enough to notice that
	$$
	S_m(F)=\frac{1}{md_m}\sum_{i=1}^{\tau_m(F)}\dim H^0(Y,m\pis L-iF)\leq\frac{1}{md_m}\sum_{i=1}^{\tau_m(F)}(d_m-1)=\frac{d_m-1}{md_m}\cdot \tau_m(F).
	$$
	Here we used the fact that $\dim H^0(Y,m\pis L-iF)<\dim H^0(X,mL)$ for all $i>0$, as $|mL|$ is base point free by \er{eq:m-divisible-mL-ample}. As a consequence,
	$$
	\delta_m(L)=\inf_F\frac{A_X(F)}{S_m(F)}\ge\delta(1-\varepsilon)+\delta\varepsilon\cdot\frac{d_m}{d_m-1}=\delta+\frac{\delta\varepsilon}{d_m-1}>\delta,
	$$
concluding the proof of Proposition \ref{prop:F-coerc=delta-m}.
\end{proof}

We record the following analytic description of $\delta_m(L)$ alluded
to in the Introduction.

\begin{corollary}
Let $L$ be an ample $\QQ$-line bundle and assume \er{eq:m-divisible-mL-ample}.
Then,	
$$
	\delta_m(L)=\sup\Big\{\delta>0\,:\,F_m^{f,\delta}\text{ is coercive on }\cB_m\Big\}.
$$
\end{corollary}

\subsection{$\theta$-balanced metrics}
\label{sec:alpha-balanced}

In this part we consider a special kind of $(f,\delta)$-balanced metric, which we will refer to as the $\theta$-balanced metric. The goal then is to prove Theorem \ref{thm:finite-dim-YTD}.

 The setup is as follows. 
Let $L$ be an ample $\QQ$-line bundle and assume \er{eq:m-divisible-mL-ample}.
As before, we fix some smooth Hermitian metric $h$ on $L$ with positive curvature form $\omega$ as the background K\"ahler form on $X$. Let $\theta$
be a smooth form cohomologous to $c_1(X)-c_1(L)$. The $\theta$-Ricci potential $f_\theta$ is the unique smooth function on $X$ satisfying
 \begin{equation}
 \label{eq:def-f-alpha}
 	\Ric(\omega)=\omega+\theta+\ddc f_\theta,\quad \int_Xe^{f_\theta}\omega^n=
\int_X\omega^n.
 \end{equation}
Define the quantized $\theta$-Ding energy on $\cB_m$:
\begin{equation*}
	F^\theta_m(\varphi):=-\log\frac{1}{V}\int_Xe^{f_\theta-\varphi}\omega^n-E_m(H,\varphi).
\end{equation*}

\begin{definition}
	Any critical point $\varphi\in\cB_m$ of $F^\theta_m$ is called $\theta$-balanced of level $m$, and $\omega_\varphi:=\omega+\ddc\varphi$ is called a $\theta$-balanced metric of level $m$.
\end{definition}

It is straightforward to verify that  the definition of $\theta$-balanced metric does not depend on the choice of the background K\"ahler form $\omega$ representing $c_1(L)$. More specifically, if $\varphi$ is $\theta$-balanced with respect to $\omega$, then given another background K\"ahler form $\omega^\prime:=\omega+\ddc\phi$, $\varphi-\phi$ is $\theta$-balanced with respect to $\omega^\prime$. 
\begin{remark}
If
$
\varphi=\frac{1}{m}\log\sum_{i=1}^{d_m}|\sigma_i|^2_{h^m}\in\cB_m
$
is $\theta$-balanced, then $\varphi$ satisfies
\begin{equation*}
	\frac{d_m}{\int_Xe^{f_\theta-\varphi}\omega^n}\int_Xh_\varphi^m(\sigma_i,\sigma_j)e^{f_\theta-\varphi}\omega^n=\delta_{ij},\ \forall\ 1\le i,j\le d_m.
\end{equation*}
where 
$h_\varphi:=he^{-\varphi}$
is the Fubini-Study Hermitian metric on $L$ induced by $\varphi$.\end{remark}

We can now turn to the proof of one of our main results,
Theorem \ref{thm:finite-dim-YTD}.
But before embarking on the proof we recall
the notion of Bergman geodesics.
Fix $m$.
Any two elements of $\cB_m$
$\varphi=\frac{1}{m}\log\sum_{i=1}^{d_m}|s_i|^2_{h^m}\text{ and }\psi=\frac{1}{m}\log\sum_{i=1}^{d_m}|\sigma_i|^2_{h^m}\in\cB_m
$
can be joined by a \emph{Bergman geodesic} that can be written as follows.
Diagonalize the basis so that $\sigma_i=e^{\lambda_i/2}s_i$ with $\lambda_i\in\RR$. Let $T>0$ be some parameter. Then
$$
\varphi(t):=\frac{1}{m}\log\sum_{i=1}^{d_m}e^{\lambda_it/T}|s_i|^2_{h^m},\ t\in[0,T].
$$
Then $\varphi(t)$, connecting $\varphi$ and $\psi$, is a geodesic segment with respect to the natural Riemannian structure of the homogeneous space $GL(d_m,\CC)/U(d_m)$. 
Bergman geodesics are \emph{sub-geodesics} in the pluripotential theory. More precisely,
consider the projection $\CC\times X\xrightarrow{\Phi}X$. Then,
\begin{equation}
\lb{subgeod}
    \Phi^*\omega+\ddc(\varphi_{\operatorname{Re}\tau})\ge0,
\end{equation}
where $\tau\in\CC$ is a complex variable.
Another ingredient we shall recall comes from \cite{DLR19}, where a well-behaved metric structure on $\cB_m$ is defind by the quantizing Darvas' $d_1$-distance. 
\begin{definition}
	For any $\varphi_1=\frac{1}{m}\log\sum_{i=1}^{d_m}|s_i|^2_{h^m},\varphi_2=\frac{1}{m}\log\sum_{i=1}^{d_m}|\sigma_i|^2_{h^m}\in\cB_m$, the \mthit\ quantum rooftop envelope $P_m(\varphi_1,\varphi_2)$ is defined in the following way. Up to unitary transformation, one can diagonalize the basis so that	$\sigma_i=\mu_is_i$ with $\mu_i>0$. Define	
	$$
	P_m(\varphi_1,\varphi_2):=\frac{1}{m}\log\sum_{i=1}^{d_m}
\min\{1,\mu_i\}|s_i|^2_{h^m}.
	$$
\end{definition}

\begin{definition}
	For any two $\varphi_1,\varphi_2\in\cB_m$, the \mth\ quantum $d_1$-distance between them is given by
	$$
	d_{1,m}(\varphi_1,\varphi_2):=E_m(H,\varphi_1)+E_m(H,\varphi_2)-2E_m(H,P_m(\varphi_1,\varphi_2)).
	$$
\end{definition}

\noindent
More concretely, diagonalizing
$$
\varphi_1=\frac{1}{m}\log\sum_{i=1}^{d_m}|s_i|^2_{h^m},\quad
\varphi_2=\frac{1}{m}\log\sum_{i=1}^{d_m}\mu_i|s_i|^2_{h^m}\in\cB_m,
$$
one has
\beq
\lb{d1mmu}
d_{1,m}(\varphi_1,\varphi_2)=\frac{1}{md_m}\sum_{i=1}^{d_m}|\log\mu_i|.
\eeq
From this, one obtains the following compactness principle.
\begin{lemma}
\label{lem:d-1-cptness}
	For any $\varphi_0\in\cB_m$ and $C>0$, the subset
	$$
	\cC:=\Big\{\varphi\in\cB_m\,:\,d_{1,m}(\varphi_0,\varphi)\le C\Big\}
	$$
	is compact with respect to the $C^\infty$-topology.
\end{lemma}

\begin{proof}
Let
$
H_0:=\operatorname{FS}^{-1}(\varphi_0).
$
Then any $\varphi\in\cB_m$ can be written as
	$
	\varphi=\frac{1}{m}\log\sum_{i=1}^{d_m}\mu_i|s_i|^2_{h^m}
	$
	for some $H_0$-orthonormal basis $\{s_i\}$ with $\mu_i>0$.
	The condition $d_{1,m}(\varphi_0,\varphi)\le C$ then implies that  there exists $C^\prime>0$ such that
	$$
	\frac{1}{C^\prime}\le\mu_i\le C^\prime,\ \forall\,1\le i\le d_m.
	$$
	Namely the eigenvalues of $\operatorname{FS}^{-1}(\varphi)$ are
contained in a compact set in $(0,\infty)$.
Also note that the set of all the $H_0$-orthonormal bases is compact. So the assertion follows.
\end{proof}

\begin{lemma}
\label{lem:d-1-m=sup-E}
Given $H\in\cP_m$,
	let
	$
	\varphi_0=\operatorname{FS}(H).
	$
	There exists $C=C(H)>0$ such that for any $\varphi\in\cB_m$ with $E_m(H,\varphi)=0$,
	$$
	\frac{1}{2}d_{1,m}(\varphi_0,\varphi)-C\le\sup\varphi-E_m(H,\varphi)\le \frac{d_m}{2} d_{1,m}(\varphi_0,\varphi)+C.
	$$
\end{lemma}

Lemma \ref{lem:d-1-m=sup-E} suggests that $\sup\varphi-E_m(H,\varphi)$, as a scaling invariant function on $\cB_m$, behaves like a distance function. This explains
its appearance in the coercivity
thresholds on $\cB_m$ (see, e.g., Definition \ref{def:coerc-F}).

\begin{proof}
	Any $\varphi\in\cB_m$ with $E_m(H,\varphi)=0$ can be written as
	$$
	\varphi=\frac{1}{m}\log\sum_{i=1}^{d_m}\mu_i|s_i|^2_{h^m},
	$$
	where $\{s_i\}$ is $H$-orthonormal so that $\varphi_0=\frac{1}{m}\log\sum_{i=1}^{d_m}|s_i|^2_{h^m}$ and 
	$
	\prod_{i=1}^{d_m}\mu_i=1.
	$
	Thus,
	$$
	\prod_{i=1}^{d_m}\min\{1,\mu_i\}=\frac{1}{\prod_{i=1}^{d_m}\max\{1,\mu_i\}}.
	$$
	Thus by definition,
	$$
	d_{1,m}(\varphi_0,\varphi)=-\frac{2}{md_m}\log\prod_{i=1}^{d_m}\min\{1,\mu_i\}=\frac{2}{md_m}\log\prod_{i=1}^{d_m}\max\{1,\mu_i\}.
	$$
	Now observe that
	$$
	\log\prod_{i=1}^{d_m}\max\{1,\mu_i\}\ge\log\mu_{\maxop}^2=\frac{1}{d_m}\log\prod_{i=1}^{d_m}\frac{\mu_{\maxop}}{\mu_{i}},
	$$
	and also
	$$
	\log\prod_{i=1}^{d_m}\frac{\mu_{\maxop}}{\mu_i}=\log\prod_{i=1}^{d_m}\mu^2_{max}\ge\log\prod_{i=1}^{d_m}\max\{1,\mu_i\}.
	$$
	So we get that
	$$
	\frac{1}{2}d_{1,m}(\varphi_0,\varphi)\le\frac{1}{md_m}\log\prod_{i=1}^{d_m}\frac{\mu_{\maxop}}{\mu_i}\le\frac{d_m}{2}d_{1,m}(\varphi_0,\varphi).
	$$
	Then the assertion follows from Lemma \ref{lem:mu-max=sup-E}.
\end{proof}

\begin{proof}[Proof of Theorem \ref{thm:finite-dim-YTD}]
	By Proposition \ref{prop:F-m-coercive=>balanced} and \ref{prop:F-coerc=delta-m}, the first part follows immediately. So it remains to show the second part. The key property we need to use is the Berndtsson convexity \cite[Section 7]{Bern15}, which implies that, under the assumption $\theta\ge0$, the functional $F^\theta_m$ is convex along Bergman geodesics (this
	uses \er{subgeod}). 
	This in particular, implies that $\theta$-balanced metric has to minimize $F_m^\theta$. So the existence of an $\theta$-balanced metric implies that
	$$
	F^\theta_m(\varphi)\ge-C,\ \forall\varphi\in\cB_m
	$$
for some $C>0$. Then by Definition \ref{def:delta-A-m}, $\delta^A_m(L)\ge1$. So $\delta_m(L)\ge1$ by Theorem \ref{thm:delta-A-m=delta-m}.

Now assume further that $\varphi_0\in\cB_m$ is the unique (up to constant) $\theta$-balanced potential. The goal is to show that $\delta_m(L)>1$. By Proposition \ref{prop:F-coerc=delta-m}, it amounts to proving that $F^\theta_m$ is coercive on $\cB_m$. To this end, we use the argument of \cite{DR17}. We choose the reference Hermitian product $H$ to be
$$
H:=\operatorname{FS}^{-1}(\varphi_0).
$$
Consider the subspace
$$
\cR:=\bigg\{\varphi\in\cB_m\,:\,E_m(H,\varphi)=0\bigg\}\subseteq\cB_m.
$$
So $\varphi_0\in\cR$ and any two elements in $\cR$ can be connected by a Bergman geodesic that is contained in $\cR$. Moreover, $\varphi_0$ is the unique minimizer of $F^\theta_m$ in $\cR$. We define
$$
\lambda:=\inf\bigg\{\frac{F^\theta_m(\varphi)-F^\theta_m(\varphi_0)}{d_{1,m}(\varphi_0,\varphi)}\,:\,\varphi\in\cR,\ d_{1,m}(\varphi_0,\varphi)\ge2\bigg\}.
$$
We claim that $\lambda>0$. Assume to the contrary that $\lambda=0$. Then there exists $\{\varphi_i\}_{i\in\NN}\subset\cR$ such that
$$
\frac{F^\theta_m(\varphi_i)-F^\theta_m(\varphi_0)}{d_{1,m}(\varphi_0,\varphi_i)}\rightarrow0.
$$
Let $[0,d_{1,m}(\varphi_0,\varphi_i)]\ni t\mapsto\varphi_{i,t}$ be the unit speed Bergman geodesic joining $\varphi_0$ to $\varphi_i$.
Then by convexity,
$$
0\le F^\theta_m(\varphi_{i,1})-F^\theta_m(\varphi_0)\le\frac{F^\theta_m(\varphi_i)-F^\theta_m(\varphi_0)}{d_{1,m}(\varphi_0,\varphi_i)}\rightarrow0.
$$
Note that $d_{1,m}(\varphi_0,\varphi_{i,1})=1$. Then by Lemma \ref{lem:d-1-cptness}, we can extract a limit
$
\varphi_{i,1}\xrightarrow{C^\infty}\varphi_\infty\in\cR
$
with $F^\theta_m(\varphi_\infty)=F^\theta_m(\varphi_0)$ and $d_{1,m}(\varphi_0,\varphi_\infty)=1$. This contradicts the uniqueness of the minimizer. So $\lambda>0$ as claimed.

Then by compactness (see Lemma \ref{lem:d-1-cptness}), we can find $C>0$ such that
$$
F^\theta_m(\varphi)\ge\lambda d_{1,m}(\varphi_0,\varphi)-C,\ \forall\varphi\in\cR.
$$
So by Lemma \ref{lem:d-1-m=sup-E}, there exist $\lambda^\prime>0$ and $C^\prime>0$ such that
$$
F^\theta_m(\varphi)\ge\lambda^\prime (\sup\varphi-E_m(H,\varphi))-C^\prime,\ \forall\varphi\in\cR.
$$
Note that this inequality is invariant under translation, so it holds for any $\varphi\in\cB_m$.
Therefore $\delta_m(L)>1$ by Proposition \ref{prop:F-coerc=delta-m}. 
This completes the proof of 
Theorem \ref{thm:finite-dim-YTD}.
\end{proof}

\section{Limiting behavior}
\label{sec:limit-behavior}
Our previous results were concerned with $m\in\NN$. In this section 
we study the limit $m\rightarrow\infty$. In this context it is convenient to use the functional 
\begin{equation}
    \label{eq:def-E-m}
    E_m(\varphi):=E_m(H_m,\varphi),
\end{equation}
where $H_m:=\int_Xh^m(\,\cdot\,,\,\cdot\,)\omega^n\in\cP_m$ is the Hermitian product induced by $h$.
To begin with,
we recall the following approximation result proved by Donaldson \cite[\S3]{D05}.
\begin{lemma}
\label{lem:lim-E-m=E}
	Given any $\varphi\in\cH_\omega$, let
	$
	\varphi_m:=\frac{1}{m}\log\sum_{i=1}^{d_m}|\sigma_i|^2_{h^m},
	$
	where $\{\sigma_i\}$ is any orthonormal basis of the following $L^2$-inner product
	$
	\int_X(he^{-\varphi})^m(\,\cdot\,,\,\cdot\,)\omega^n_\varphi.
	$
	Then,
	$
	E_m(\varphi_m)\rightarrow E(\varphi) \text{ as }m\rightarrow\infty.
	$
\end{lemma}

Also recall the following result (cf. \cite[Lemma 7.7]{BBGZ}).
\begin{lemma}
\label{lem:sup-E<1+e}
	There exists $\varepsilon_m\rightarrow0$ such that  for all $m\gg1$,
	$$
	\sup\varphi-E(\varphi)\le(1+\varepsilon_m)(\sup\varphi-E_m(\varphi))+\varepsilon_m\text{ on }\cB_m. 
	$$
\end{lemma}

\begin{proof}
	 For any $\varphi\in\cB_m$, we may write
	$
	\varphi=\frac{1}{m}\log\sum_{i=1}^{d_m} e^{\lambda_i}|s_i|^2_{h^m}
	$
	for some $H_m$-orthonormal basis $\{s_i\}$ and $\lambda_i\in\RR$.
	Set
	$
	\lambda_{max}:=\max_i\{\lambda_i\}
	$
	and
	$
	\varphi(t):=\frac{1}{m}\log\sum_{i=1}^{d_m} e^{\lambda_it}|s_i|^2_{h^m},\ t\ge0.
	$
	Note that  $E$ is convex along Bergman geodesics (cf. \cite[Proposition 1]{D05}).
	Thus
	\begin{equation*}
		\begin{aligned}
			E(\varphi)&=E(\varphi_1)-E(\varphi_0)+E(\varphi_0)\\
			&\ge\frac{d}{dt}\bigg|_{t=0}E(\varphi_t)+E(\varphi_0)\\
			&=\frac{1}{V}\int_X\dot{\varphi}_0\omega^n_{\varphi_0}+E(\varphi_0)\\
			&=\frac{1}{mV}\int_X\frac{\sum_{i=1}^{d_m}\lambda_i|s_i|^2_{h^m}}{\sum_{i=1}^{d_m}|s_i|^2_{h^m}}\omega^n_{\varphi_0}+E(\varphi_0)\\
			&=\frac{1}{mV}\int_X\frac{\sum_{i=1}^{d_m}(\lambda_i-\lambda_{max})|s_i|^2_{h_m}}{\rho_m}\omega^n_{\frac{1}{m}\log\rho_m}+\frac{\lambda_{max}}{m}+E(\frac{1}{m}\log\rho_m),
		\end{aligned}
	\end{equation*}
where
$\rho_m:=\sum_{i=1}^{d_m}|s_i|^2_{h^m}$.
For each $m$ as in \ref{eq:m-divisible-mL-ample} let $\{\sigma_i\}_{i=1}^{d_m}$ be an orthonormal
basis with respect to $L^2$ inner product on $H^0(X,mL)$ induced by $h$.
By a classical theorem of Catlin, Ruan, Tian, Zelditch \cite{Cat99,Rua98,Tian89,Zel90}, 
\begin{equation}
\label{eq:rho/d->1/V}
	\frac{\rho_m
}{d_m}\xrightarrow{C^\infty}\frac{1}{V},
\end{equation}
implying that
\begin{equation}
\label{eq:log-rho->0}
	\cB_m\ni\frac{1}{m}\log\rho_m
\xrightarrow{C^\infty}0.
\end{equation}
Using \eqref{eq:rho/d->1/V} and \eqref{eq:log-rho->0}, one finds $\varepsilon_m\rightarrow0$ (independent of $\varphi\in\cB_m$) such that
\begin{equation*}
	\begin{aligned}
		E(\varphi)&\ge\frac{1+\varepsilon_m}{md_m}\int_X\sum_{i=1}^{d_m}(\lambda_i-\lambda_{max})|s_i|^2_{h^m}\omega^n+\frac{\lambda_{max}}{m}-\varepsilon_m\\
		&=\frac{1+\varepsilon_m}{md_m}\sum_{i=1}^{d_m}(\lambda_i-\lambda_{max})+\frac{\lambda_{max}}{m}-\varepsilon_m\\
		&=(1+\varepsilon_m)\bigg(E_m(\varphi)-\frac{\lambda_{max}}{m}\bigg)+\frac{\lambda_{max}}{m}-\varepsilon_m.\\
	\end{aligned}
\end{equation*}
Then the desired inequality follows from Lemma \ref{lem:sup=max-mu}.
\end{proof}

\subsection{Estimating $\delta$-invariant}
Firstly, we take the opportunity to give a more direct proof of \cite[Proposition 3.11]{Zha20} without relying on the non-Archimedean approach of \cite{BBJ18}. 
\begin{proposition}
\label{prop:delta-A<=delta}
	For any ample $\QQ$-line bundle $L$, one has
	$$
	\delta(L)\ge\delta^A(L).
	$$
\end{proposition}

\begin{proof}
	Applying H\"older inequality as in the proof of Proposition \ref{prop:F-coerc=delta-m}, we derive that  for any $\delta\in(0,\delta^A(L))$, there exist $\lambda\in(0,1)$ and $C>0$ such that
	\begin{equation}
	\label{eq:delta-Ding-coercive}
		-\frac{1}{\delta}\log\int_Xe^{-\delta\varphi}\omega^n-E(\varphi)\ge\lambda(\sup\varphi-E(\varphi))-C,\ \forall\varphi\in\cH_\omega.
	\end{equation}
	Then by Lemma \ref{lem:sup-E<1+e}, we can find $\lambda^\prime>0$ and $C^\prime>0$ such that
	$$
	-\frac{1}{\delta}\log\int_Xe^{-\delta\varphi}\omega^n-E_m(\varphi)\ge\lambda^\prime(\sup\varphi-E_m(\varphi))-C^\prime,\ \forall\varphi\in\cB_m
	$$
	for $m\gg1$. Then Proposition \ref{prop:F-coerc=delta-m} implies that
	$$
	\delta_m(L)>\delta.
	$$
	Thus \cite[Theorem B]{BJ17} implies that
	$
	\delta(L)\ge\delta,
	$
	finishing the proof.
\end{proof}

The next result, improving Proposition \ref{prop:F-coerc=delta-m}, provides a more accurate estimate for the $\delta$-invariant.
\begin{proposition}
	Assume that there exist $\delta>0$, $\lambda>0$, $m_j\rightarrow\infty$ and $C_j>0$ such that	$$
	-\frac{1}{\delta}\log\int_Xe^{-\delta\varphi}\omega^n-E_{m_j}(\varphi)\ge\lambda(\sup\varphi-E_{m_j}(\varphi))-C_j,\ \forall\varphi\in\cB_{m_j}.
	$$
	Then one has
	$$
	\delta(L)\ge\delta(1+\frac{\lambda}{n}).
	$$
\end{proposition}

\begin{proof}
	We proceed as in the proof of Proposition \ref{prop:F-coerc=delta-m}. For each $m_j$, one has
	$$
	\delta(1-\lambda)S_{m_j}(F)+\delta\lambda \tau_{m_j}(F)/m_j-A_X(F)\le0,\ \text{for all $F$ over }X.
	$$
	Set
	\begin{equation*}
		\begin{cases}
			S(F):=\lim_{j}S_{m_j}(F),\\
			\tau(F):=\lim_{j}\tau_{m_j}(F)/m_j.\\
		\end{cases}
	\end{equation*}
	Then we can find $\varepsilon_j\rightarrow0$ such that \cite[Corollary 3.6]{BJ17}
	$$
	S_{m_j}(F)\le(1+\varepsilon_j)S(F),\ \text{for all $F$ over }X.
	$$
	Moreover by \cite[Theorem 5.1]{BJ17}, there exists $C>0$ such that
	$$
	0\le \tau(F)-\tau_{m_j}(F)/m_j\le\frac{CA_X(F)}{m_j},\ \text{for all $F$ over }X.
	$$
	Thus we deduce that
	$$
	\delta(1-\lambda)S_{m_j}(F)+\delta\lambda\bigg(\tau(F)-\frac{CA_X(F)}{m_j}\bigg)-A_X(F)\le0,\ \text{for all $F$ over }X.
	$$
	Let $F_j$ be the divisor computing $\delta_{m_j}(L)$. Then,
	\begin{equation*}
		\begin{aligned}
			\bigg(1+\frac{C\delta\lambda}{m_j}\bigg)\delta_{m_j}(L)
			&\ge\delta-\delta\lambda+\delta\lambda\frac{\tau(F_j)}
{S_{m_j}(F_j)}
			\ge\delta-\delta\lambda+\frac{\delta\lambda}{1+\varepsilon_j}\cdot\frac{\tau(F_j)}{S(F_j)}.
		\end{aligned}
	\end{equation*}
By \cite[Proposition 2.1]{Fuj19b},
$
\frac{\tau(F)}{S(F)}\ge\frac{n+1}{n},\ \text{for all $F$ over }X,
$
so
$$
\bigg(1+\frac{C\delta\lambda}{m_j}\bigg)\delta_{m_j}(L)\ge\delta-\delta\lambda+\frac{\delta\lambda}{1+\varepsilon_j}\cdot\frac{n+1}{n}.
$$
Sending $m_j$ to $\infty$ and $\varepsilon_j$ to $0$
concludes the proof.
\end{proof}

A direct consequence is the following estimate of $\delta(L)$.

\begin{corollary}
	Assume that for some $\delta>0$, $\lambda>0$ and $C>0$, it holds that
	$$
		-\frac{1}{\delta}\log\int_Xe^{-\delta\varphi}\omega^n-E(\varphi)\ge\lambda(\sup\varphi-E(\varphi))-C,\ \forall\varphi\in\cH_\omega.
    $$
    Then we have
    $$
    \delta(L)\ge\delta(1+\frac{\lambda}{n}).
    $$
\end{corollary}

\begin{proof}
	Applying Lemma \ref{lem:sup-E<1+e}, we may find $m_j\rightarrow\infty$, $\lambda_j\rightarrow\lambda$ and $C_j\rightarrow C$ such that
		$$
		-\frac{1}{\delta}\log\int_Xe^{-\delta\varphi}\omega^n-E_{m_j}(\varphi)\ge\lambda_j(\sup\varphi-E_{m_j}(\varphi))-C_j,\ \forall\varphi\in\cB_{m_j}.
    $$
    Then the previous proof leads us to
    $$
\bigg(1+\frac{C\delta\lambda_j}{m_j}\bigg)\delta_{m_j}(L)\ge\delta-\delta\lambda_j+\frac{\delta\lambda_j}{1+\varepsilon_j}\cdot\frac{n+1}{n}.
$$
Taking the limit, we complete the proof.
\end{proof}
As a consequence, we have reproduced the following statement in \cite{BBJ18}, without using non-Archimedean tools.
\begin{corollary}
\lb{DingCor}
	 If the  Ding functional $F^{0,1}$
\er{Dingfdelta}
 is coercive (i.e., \eqref{eq:delta-Ding-coercive} holds for $\delta=1$), then $\delta(L)>1$.
\end{corollary}

Note that Remark \ref{anyfRem} is used in Corollary \ref{DingCor}.

\subsection{Approximating twisted K\"ahler--Einstein metrics}

In what follows, we will prove several approximation results, exploring the relation between $\theta$-balanced metrics and $\theta$-\KE metrics. We emphasize that the twist term $\theta$ is {\it not } assumed to be semi-positive (otherwise $X$ would have to be Fano, which is a very restrictive assumption).
The following is an analogue of Donaldson's result \cite[Theorem 2]{D05}.
\begin{proposition}
\label{prop:tKE-appro-smoothly-by-balanced}
Let $L$ be an ample $\QQ$-line bundle and assume \er{eq:m-divisible-mL-ample}.
Let $\theta$ be a smooth form in $c_1(X)-c_1(L)$.
	Assume that for some $m_j\rightarrow\infty$, there exists a sequence of $\theta$-balanced metric $\omega_j$ in $\cB_{m_j}$ converging smoothly to a limit $\omega_\infty$. Then
	$
	\Ric(\omega_\infty)=\omega_\infty+\theta.
	$
\end{proposition}

\begin{proof}
	Let $h_\infty$ be a smooth Hermitian metric on $L$ with $
\omega_\infty$ as its curvature form. Also set $f_\infty$ to be the unique real valued function satisfying
$$
\Ric(\omega_\infty)=\omega_\infty+\theta+\ddc f_\infty,\ \int_Xe^{f_\infty}\omega_\infty^n=V.
$$
Since $\omega_j$ is $\theta$-balanced, we may write
$
\omega_j=\omega_\infty+\ddc\varphi_j,
$
where $\varphi_j$ takes the form
$$
\varphi_j=\frac{1}{m_j}\log\sum_{i=1}^{d_{m_j}}|\sigma^{(j)}_i|^2_{h_\infty^{m_j}}
$$
for some orthonormal basis $\{\sigma^{(j)}_i\}$ of the following $L^2$-inner product on $H^0(X,m_jL)$:
$$
\frac{d_{m_j}}{\int_Xe^{f_\infty-\varphi_j}\omega_\infty^n}\int_X(h_\infty e^{-\varphi_j})^{m_j}(\,\cdot\,,\,\cdot\,)e^{f_\infty-\varphi_j}\omega^n_\infty.
$$
Since $\omega_j\xrightarrow{C^\infty}\omega_\infty$, we may suitably normalize $\varphi_j$ so that
$
\varphi_j\xrightarrow{C^\infty}0.
$
Now consider the Bergman kernel
$$
\eta_{m_j}:=\sum_{i=1}^{d_{m_j}}|\sigma^{(j)}_i|^2_{(h_\infty e^{-\varphi_{j}})^{m_j}}.
$$
On the one hand, it is a constant (equal to 1) by the balanced assumption. On the other hand, by 
\cite{Cat99,Zel90},\cite[Theorem 1.3]{DLM06} 
$
\eta_{m_j}\xrightarrow{}e^{-f_\infty}.
$
This forces $f_\infty$ to be zero so that $\omega_\infty$ satisfies 
$
	\Ric(\omega_\infty)=\omega_\infty+\theta.
	$, as claimed.
\end{proof}

Next we show that, under certain reasonable assumptions, $\theta$-\KE metrics can be approximated by $\theta$-balanced metrics. Such kind of result, first appearing in \cite{Don01}, has been studied extensively by many authors over the years. The version we prove below follows closely the exposition in \cite[Section 7]{BBGZ}, with
some modifications using Darvas' work \cite{Dar15,Dar19}.
\begin{proposition}
\label{prop:tKE-approx-by-alpha-balanced}
Let $L$ be an ample $\QQ$-line bundle and assume \er{eq:m-divisible-mL-ample}. Assume that $\delta^A(L)>1$.
Let $\theta$ be a smooth form in $c_1(X)-c_1(L)$.
	Assume that $\omega$ is the unique K\"ahler form in $ c_1(L)$ solving
	$
	\Ric(\omega)=\omega+\theta.
	$
	Then there exists a sequence of $\theta$-balanced $\varphi_j\in\cB_{m_j}$ 
for some $m_j\rightarrow\infty$ such that
	$$
	\varphi_j\rightarrow 0 \quad \hbox{in the $d_1$-topology.}
	$$
	
\end{proposition}

\begin{proof}
Since $\delta^A(L)>1$, applying H\"older inequality as in the proof of Proposition \ref{prop:F-coerc=delta-m}, we find that
	$$
	F^\theta(\varphi):=-\log\frac{1}{V}\int_Xe^{-\varphi}\omega^n-E(\varphi)\ge\lambda(\sup\varphi-E(\varphi))-C,\quad \forall\varphi\in\cH_\omega,
	$$
	for some $\lambda\in(0,1)$ and $C>0$. In other words, the \emph{$\theta$-twisted Ding functional} is coercive on $\cH_\omega$ (in this setting the $\theta$-twisted Ricci potential $f_\theta$ defined by \eqref{eq:def-f-alpha} is equal to 0). So $F^\theta$ admits a minimizer $\varphi\in\cH_\omega$, which gives rise 
to a $\theta$-\KE metric. Then by our uniqueness assumption, one must have $\varphi=\operatorname{const}$, so that
	$$
	F^\theta(\varphi)\ge F^\theta(0)=0,\ \forall\varphi\in\cH_\omega.
	$$
	Let $m_j\rightarrow\infty$ be a sequence of sufficiently divisible integers.
	Our goal is then to show that there exists a sequence of $\theta$-balanced $\varphi_j\in\cB_{m_j}$ such that
	$
	\varphi_j\rightarrow0
	$
	in a suitable sense. 
	
	By the coercivity of $F^\theta$ and Lemma \ref{lem:sup-E<1+e}, there exist $\lambda^\prime>0$ and $C^\prime>0$ such that,
	$$
	F^\theta_{m_j}(\varphi)
	=-\log\frac{1}{V}\int_Xe^{f_\theta-\varphi}\omega^n-E_{m_j}(\varphi)
	\ge\lambda^\prime(\sup\varphi-E_{m_j}(\varphi))-C^\prime,\ \forall\varphi\in\cB_{m_j}.
	$$
	So by Proposition \ref{prop:F-m-coercive=>balanced},
	we can find $\theta$-balanced metric $\varphi_j\in\cB_{m_j}$, minimizing $F^\theta_{m_j}$. Thus, using \eqref{eq:log-rho->0} 
 and Lemma \ref{lem:lim-E-m=E} proven below,
	$$
	\lambda^\prime(\sup\varphi_j-E_m(\varphi_j))-C^\prime\le F^\theta_{m_j}(\varphi_j)\le F^\theta_{m_j}\bigg(\frac{1}{m_j}\log
\sum_{i=1}^{d_{m_j}}|\sigma_i|^2_{h^{m_j}}
\bigg)\xrightarrow{m_j\rightarrow\infty}F^\theta(0)=0.
	$$
	This implies that
	$$
	\sup\varphi_j-E_{m_j}(\varphi_j)<A,\ \forall m_j\gg1,
	$$
	for some $A>0$. Then by Lemma \ref{lem:sup-E<1+e}, we derive that
	$
	\sup\varphi_j-E(\varphi_j)<A
	$
	for $m_j\gg1$. 
		We may normalize each $\varphi_j$ so that 
	$\sup\varphi_j=0.$
	Then we have
	$$
	d_1(0,\varphi_j)=-E(\varphi_j)<A
	$$
	for $m_j\gg1$. Here we are using the $d_1$-distance of Darvas \cite{Dar19}.

	Now by Lemma \ref{lem:lim-E-m=E}, Lemma \ref{lem:sup-E<1+e} and \eqref{eq:log-rho->0} again,
	\begin{equation*}
		\begin{aligned}
			0\le F^\theta(\varphi_j)&=-\log\frac{1}{V}\int_Xe^{f-\varphi_j}\omega^n-E(\varphi_j)\\
			&\le-\log\frac{1}{V}\int_Xe^{f-\varphi_j}\omega^n-(1+\varepsilon_{m_j})E_{m_j}(\varphi_j)+\varepsilon_{m_j}\\
			&=F^\theta_{m_j}(\varphi_j)-\varepsilon_{m_j} E_{m_j}(\varphi_j)+\varepsilon_{m_j}\\
			&\le F^\theta_{m_j}(\frac{1}{m_j}\log\rho_{m_j})+\varepsilon_{m_j} A+\varepsilon_{m_j}\xrightarrow{m_j\rightarrow\infty}0.
		\end{aligned}
	\end{equation*}
	Therefore, $
	\lim_{j} F^\theta(\varphi_j)=0=\inf F^\theta.
	$
	In other words, $\{\varphi_j\}$ is a $d_1$-bounded minimizing sequence of the $\theta$-twisted Ding functional $F^\theta$. Then one can extract a subsequence so that $\varphi_j$ $d_1$-converges to a minimizer of $F^\theta$ (cf. the proof of \cite[Theorem 4.18]{Dar19}). By regularity theory \cite{ST11,BBEGZ19}, this yields an $\theta$-twisted KE potential, which has to be a constant by uniqueness assumption, and hence which has to be 0 as $\sup\varphi_j=0$. Thus,
	$
	d_1(0,\varphi_j)\rightarrow0.
	$
\end{proof}

\begin{remark}{\rm
	A more delicate treatment following \cite{Don01,IO20} can possibly improve the regularity of the convergence in the above result. We leave this to the interested readers.}
\end{remark}

Subtlety arises when the $\theta$-twisted KE metric is not unique, which is usually due to the existence of non-trivial holomorphic vector fields. In this case, it is not expected that one can approximate the $\theta$-twisted KE by $\theta$-balanced metrics. However, the next result shows that  one can still get a satisfactory approximation if the twist term is allowed to be perturbed (by considering a sequence of slightly perturbed Ding functional).

\begin{proposition}
\label{prop:alpha-approx-by-0-delta-i}
	Let $L$ be an ample $\QQ$-line bundle
 and assume \er{eq:m-divisible-mL-ample}. Let $\theta$
 be a smooth form 
in $c_1(X)-c_1(L)$. Assume that $\omega$ satisfies
	$
	\Ric(\omega)=\omega+\theta.
	$
Moreover, assume that
$\varphi=0$ is a minimizer of the following twisted Ding energy
	$$
	F^\theta(\varphi):=-\log\frac{1}{V}\int_Xe^{-\varphi}\omega^n-E(\varphi),\ \varphi\in\cH_\omega.
	$$
Then we can find $m_j\rightarrow\infty$, $\delta_j\nearrow1$ and a sequence of $(0,\delta_i)$-balanced $\varphi_j\in\cB_{m_j}$ such that
		$$
	\varphi_j\rightarrow 0 \quad \hbox{in the $d_1$-topology.}
	$$
\end{proposition}

\begin{proof}
By our assumption,
	$
	F^\theta(\varphi)\ge F^\theta(0)=0,
	$
	thus
	$
	\delta^A(L)\ge1.
	$
	Pick any strictly increasing sequence $\delta_j\nearrow1.$ Then for each $j$, applying H\"older inequality as in the proof of Proposition \ref{prop:F-coerc=delta-m}, we can find $\lambda_j\in(0,1)$ and $C_j>0$ such that
	$$
	-\frac{1}{\delta_j}\log\frac{1}{V}\int_Xe^{-\delta_j\varphi}\omega^n-E(\varphi)\ge\lambda_j(\sup-E(\varphi))-C_j,\ \forall\varphi\in\cH_\omega.
	$$
	Then by Lemma \ref{lem:sup-E<1+e}, after perturbing slightly $\lambda_j$ and $C_j$, we deduce that
	$$
	F^{0,\delta_j}_m(\varphi)=-\frac{1}{\delta_j}\log\frac{1}{V}\int_Xe^{-\delta_j\varphi}\omega^n-E_m(\varphi)\ge\lambda_j(\sup-E_m(\varphi))-C_j,\ \forall\varphi\in \cB_m
	$$
	for all sufficiently divisible $m\gg1$. 
	Then by Proposition \ref{prop:F-m-coercive=>balanced} and the proof of Proposition \ref{prop:tKE-approx-by-alpha-balanced}, there exists $(0,\delta_j)$-balanced	$\varphi^j_m\in\cB_m$, minimizing $F^{0,\delta_j}_m$, such that
	$$
	\sup\varphi^j_m-E_m(\varphi^j_m)<A_j,
	$$
	for some $A_j>0$ (independent of $m$). So as $m\rightarrow\infty$, $\varphi^j_m-\sup\varphi^j_m$ is a $d_1$-bounded sequence. 
Now put
	$$
	F^{0,\delta_j}(\varphi):=-\frac{1}{\delta_j}\log\frac{1}{V}\int_Xe^{-\delta_j\varphi}\omega^n-E(\varphi),\ \varphi\in\cH_\omega.
	$$
	Then we have (by H\"older inequality, Lemma \ref{lem:sup-E<1+e}, \ref{lem:lim-E-m=E} and \eqref{eq:log-rho->0})
	\begin{equation*}
		\begin{aligned}
			0\le F^\theta(\varphi^j_m)&=-\log\frac{1}{V}\int_Xe^{-\varphi^j_m}\omega^n-E(\varphi^j_m)\\
			&\le-\frac{1}{\delta_j}\log\frac{1}{V}\int_Xe^{-\delta_j\varphi^j_m}\omega^n-E(\varphi^j_m)\\
			&\le F^{0,\delta_j}_m(\varphi^j_m)+\varepsilon_m(\sup\varphi^j_m-E_m(\varphi^j_m))+\varepsilon_m\\
			&\le F^{0,\delta_j}_m(\frac{1}{m}\log\rho_{m})+\varepsilon_m A+\varepsilon_m\\
			&\xrightarrow{m\rightarrow\infty}F^{\alpha,\delta_j}(0)=0.
		\end{aligned}
	\end{equation*}
	So for each fixed $j$, $\{\varphi^j_m-\sup\varphi^j_m\}_{m\gg1}$ is a $d_1$-bounded minimizing sequence of $F^{0,\delta_j}$. 	Then we may extract a limit $\varphi^j_\infty\in\cE^1(X,\omega)$ as in the previous proof, such that
	$$
	0\le F^\theta(\varphi^j_\infty)\le F^{0,\delta_j}(\varphi^j_\infty)=0.
	$$
	This forces $\varphi^j_\infty$ to be a minimizer of $F^\theta$,
	so that $\varphi^j_\infty\in\cH_\omega$ is an $\theta$-twisted KE potential (as in the previous proof). Moreover, the equality
	$$
	F^\theta(\varphi^j_\infty)= F^{0,\delta_j}(\varphi^j_\infty)
	$$
	forces $\varphi^j_\infty$ to be a constant, which hence must be 0 as $\sup\varphi^j_\infty=0$. Therefore, we obtain that (up to a subsequence)
	$$
	d_1(\varphi^j_m-\sup \varphi^j_m,0)\rightarrow0\ \text{as }m\rightarrow\infty.
	$$
	So we can pick a sufficiently divisible $m_j$ such that
	$$
	d_1(\varphi^j_{m_j}-\sup\varphi^j_{m_j},0)<\frac{1}{j}.
	$$
	Set
	$
	\varphi_j:=\varphi^j_{m_j},
	$
	then $\varphi_j$ is $(0,\delta_j)$-balanced and
	$
	d_1(\varphi_j-\sup\varphi_j,0)\rightarrow0\text{ as }j\rightarrow\infty.
	$
This completes the proof.
	\end{proof}
	
\begin{remark}{
\rm	Using the balanced equation of $\varphi_j$ and the proof of \cite[Theorem 4.4]{DLR19}, it is not hard to derive that  the above convergence $d_1(\varphi_j-\sup\varphi_j,0)\rightarrow0$ implies that
	$$
	\sup\varphi_j-E_{m_j}(\varphi_j)<A
	$$
	for some uniform $A>0$.}
\end{remark}

It turns out that  the converse direction of the previous proposition also holds. The result we state next can be thought of as the quantization of the classical continuity method.

\begin{proposition}
	Let $L$ be an ample $\QQ$-line bundle  and assume 
\er{eq:m-divisible-mL-ample}. Let $\theta$
be a smooth form in $c_1(X)-c_1(L)$. Let $f_\theta$ be the twisted Ricci potential function defined by \eqref{eq:def-f-alpha}. Assume that for some $m_j\rightarrow\infty$, and $\delta_j\nearrow1$, there exists a sequence of $(f_\theta,\delta_j)$-twisted balanced $\varphi_j\in\cB_{m_j}$, minimizing $F^{f_\theta,\delta_j}_{m_j}$ on $\cB_{m_j}$, such that
	$$
	\sup\varphi_j-E_{m_j}(\varphi_j)<A
	$$
	for some uniform $A>0$. Then there exists a K\"ahler form $\omega_\infty\in  c_1(L)$ solving
	$
	\Ric(\omega)_\infty=\omega_\infty+\theta.
	$
\end{proposition}

Note that $\varphi_j\in\cB_{m_j}$ being a minimizer of $F^{f_\theta,\delta_j}_{m_j}$ implies that $\delta_{m_j}(L)=\delta^A_{m_j}(L)\ge\delta_j$ (recall Theorem \ref{thm:delta-A-m=delta-m}). Thus $\delta(L)\ge1$.

\begin{proof}
Consider the Ding functional
$$
F^\theta(\varphi):=-\log\frac{1}{V}\int_Xe^{f_\theta-\varphi}\omega^n-E(\varphi),\ \varphi\in\cH_\omega,
$$
where $f_\theta$ is defined by \eqref{eq:def-f-alpha}.
We first claim that 
$$
\inf_{\varphi\in\cH_\omega}F^\theta(\varphi)>-\infty.
$$
Indeed, by Lemma \ref{lem:lim-E-m=E}, the assumption $\sup\varphi_j-E_{m_j}(\varphi_j)<A$ implies that 
$\{\varphi_j-\sup{\varphi_j}\}$ is a $d_1$-bounded sequence. Then by the Skoda estimate \cite[Theorem {5.7}]{DR17}, there exists $C>0$ such that
$$
\frac{1}{V}\int_Xe^{-(\varphi_j-\sup\varphi_j)}\omega^n<C.
$$
Therefore,
	\begin{equation*}
	\begin{aligned}
		F^\theta(\varphi_j)&=-\log\frac{1}{V}\int_Xe^{f_\theta-(\varphi_j-\sup\varphi_j)}\omega^n+\sup\varphi_j-E(\varphi_j)\\
		&\ge-\sup f_\theta-\log C=:-C^\prime.
	\end{aligned}
\end{equation*}
We then claim that
$$
\inf_{\varphi\in\cH_\omega}F^\theta(\varphi)\ge -C^\prime.
$$ 
Assume to the contrary that there exists $\psi\in \cH_\omega$ such that
$
F^\theta(\psi)<-C^\prime.
$
For each $m_j\gg1$, we put
$$
\psi_j:=\frac{1}{m_j}\log\sum_{i=1}^{d_{m_j}}|\sigma_i|_{h^{m_j}}^2,
$$
where $\{\sigma_i\}$ is any orthonormal basis of
$
\int_X(he^{-\psi})^{m_j}(\,\cdot\,,\,\cdot\,)\omega^n_\psi.
$
Then the asymptotic of Bergman kernel \eqref{eq:log-rho->0} implies that
$$
\psi_j\xrightarrow{C^\infty}\psi,\ \text{as }j\rightarrow\infty.
$$
So using Lemma \ref{lem:sup-E<1+e}, the assumption that $\varphi_j$ is a minimizer of $F^{f_\theta,\delta_j}_{m_j}$ on $\cB_{m_j}$ and Lemma \ref{lem:lim-E-m=E},  we deduce that
\begin{equation*}
	\begin{aligned}
		-C^\prime\le F^\theta(\varphi_j)&=-\log\frac{1}{V}\int_Xe^{f_\theta-\varphi_j}\omega^n-E(\varphi_j)\\
		&\le-\frac{1}{\delta_j}\log\frac{1}{V}\int_Xe^{f_\theta-\delta_j\varphi_j}\omega^n-E(\varphi_j)\\
			&\le-\frac{1}{\delta_j}\log\frac{1}{V}\int_Xe^{f-\delta_j\varphi_j}\omega^n-\sup\varphi_j+(1+\varepsilon_{m_j})(\sup\varphi_j-E_{m_j}(\varphi_j)+\varepsilon_{m_j}\\
			&=F^{f_\theta,\delta_j}_{m_j}(\varphi_j)+\varepsilon_{m_j} (\sup\varphi_j-E_{m_j}(\varphi_j)+\varepsilon_{m_j}\\
			&\le F^{f_\theta,\delta_j}_{m_j}(\psi_j)+\varepsilon_{m_j} A+\varepsilon_{m_j}\xrightarrow{j\rightarrow\infty}F^\theta(\psi)<-C^\prime,
	\end{aligned}
\end{equation*}
a contradiction. So $\inf F^\theta\ge-C^\prime$, as claimed.

We further claim that
$$
\lim_j F^\theta(\varphi_j)=\inf_{\varphi\in\cH_\omega}F^\theta(\varphi).
$$
This is clear from the above argument.
Indeed, for any $\varepsilon>0$, there exists $\psi\in\cH_\omega$ such that
$$
F^\theta(\psi)\le\inf_{\varphi\in\cH_\omega}F^\theta(\varphi)+\varepsilon,
$$
Then arguing as above, one obtains that
$$
\inf_{\varphi\in\cH_\omega}F^\theta(\varphi)\le\liminf_j F^\theta(\varphi_j)\le\limsup_j F^\theta(\varphi_j)\le\lim_j F^{f_\theta,\delta_j}_{m_j}(\psi_j)\le\inf_{\varphi\in\cH_\omega}F^\theta(\varphi)+\varepsilon.
$$
Sending $\varepsilon\rightarrow0$, we find that
$$
\lim_j F^\theta(\varphi_j)=\inf_{\varphi\in\cH_\omega}F^\theta(\varphi).
$$

Therefore $\{\varphi_j-\sup\varphi_j\}$ is a $d_1$-bounded minimizing sequence of $F^\theta$. Then one can extract a limit $\varphi_\infty\in\cE^1(X,\omega)$ such that
$$
F^\theta(\varphi_\infty)=\inf_{\varphi\in\cH_\omega}F^\theta(\varphi).
$$
By regularity theory \cite{ST11,BBEGZ19} one further has $\varphi_\infty\in\cH_\omega$ and hence $\omega_\infty:=\omega+\ddc\varphi_\infty$ is a twisted KE metric.
\end{proof}

\section{Weighted $\delta$-invariant and soliton type metrics}
\label{sec:soliton}

As we have seen, in the $\theta$-K\"ahler--Einstein setting, the above framework naturally leads us to the $\delta$-invariant.
In this section, we shall show that the same strategy extends, with some delicate
adjustments,
to soliton type metrics. This will lead us to a new, weighted, $\delta$-invariant.

\subsection{K\"ahler Ricci $g$-soliton}
We use the setup of Berman--Witt Nystr\"om \cite{BWN14}.
Let $(X,L)$ be an $n$-dimensional polarized variety admitting a Hamiltonian $T$-action, where $T\cong(S^1)^r$ is a real torus of dimension $r$ (with $r\le n$). We denote by $T_\CC\cong(\CC^*)^r$ the complexification of $T$. Here we allow $L$ to be ample $\QQ$-line bundle. Assume that $T_\CC$ acts effectively and holomorphically on $X$. Moreover assume that the action of $T_\CC$ also lifts to $L$.

We fix a $T$-invariant smooth Hermitian metric $h$ on $L$, whose curvature form will be denoted by $\omega$, which will be treated as a $T$-invariant background K\"ahler form in $ c_1(L)$. Note that $\omega$ induces a moment map
$$
m_\omega:X\rightarrow\RR^r,
$$
 whose image is a polytope 
 $$P:=m_\omega(X)\subset\RR^r,$$
that is an invariant of the class $[\omega]$.
 Let
$$
g:P\rightarrow\RR_{>0}
$$
be a smooth positive function on $P$. 

Consider
$$
\cH^T_\omega:=\bigg\{\varphi\in\cH_\omega\,:\,\varphi\text{ is $T$-invariant}\bigg\}.
$$
Then for each $\varphi\in\cH^T_\omega$, it induces a moment map
$
m_\varphi:=m_{\omega_\varphi}:X\rightarrow P,
$
so one can define the following weighted Monge-Amp\`ere measure:
$$
g\circ m_{\varphi}\,\omega_\varphi^n.
$$
Given any $T$-invariant smooth form $\theta\in(c_1(X)-c_1(L))$, let $f_\theta$ be the $T$-invariant $\theta$-twisted Ricci potential given by \eqref{eq:def-f-alpha}. Then one can consider the following Monge-Amp\`ere equation:
\begin{equation}
\label{eq:MA-equ-g-soliton}
	g\circ m_{\varphi}\,\omega_\varphi^n=e^{f_\theta-\varphi}\omega^n.
\end{equation}
A solution to this equation will give rise to an \emph{$\theta$-twisted K\"ahler-Ricci $g$-soliton}, which satisfies
$$
\Ric(\omega)_\varphi=\omega_\varphi+\theta+\ddc \log g\circ m_\varphi\,.
$$

To study this equation, a natural functional to consider is the following $g$-weighted $\theta$-Ding functional (cf. \cite{TZ02,BWN14,HL20})
\begin{equation*}
	F^{\theta,g}(\varphi):=-\log\int_Xe^{f_\theta-\varphi}\omega^n-E^g(\varphi),\ \varphi\in\cH^T_\omega,
\end{equation*}
where
\begin{equation}
\label{eq:def-E-g}
	E^g(\varphi):=E^g_{\omega}(\varphi):=\frac{1}{\int_Xg\circ m_\omega\,\omega^n}\int_0^1\int_X\varphi g\circ m_{s\varphi}\,\omega_{s\varphi}^nds
\end{equation}
is the $g$-weighted Monge-Amp\`ere energy (going back to Zhu \cite{Zhu00}). It is straightforward to show that  any critical point of $F^{\theta,g}$ 
satisfies \eqref{eq:MA-equ-g-soliton} (up to a normalization).

\begin{definition}
Set
$$
		\delta^{A,g}(L):=
\sup\bigg\{\delta>0\,:\,\sup_{\varphi\in\cH^T_\omega}\int_Xe^{-\delta(\varphi-E^g(\varphi))}\omega^n<\infty\bigg\}.
	$$
\end{definition}

\subsection{Quantization}
\label{sec:quantize-g-soliton}
Following \cite[Section 4]{BWN14}, one can quantize the above setup as follows. 
We assume \er{eq:m-divisible-mL-ample}  throughout.
By assumption, the vector space 
 $H^0(X,mL)$
 admits a $T_\CC$-action. Denote by 
 $P_m\subset\ZZ^r$
 the set of all weights of this $T_\CC$-action. It is well-known that 
 $\frac{1}{m}P_m\subseteq P$ (see e.g., \cite[Lemma 13]{Lah19}).
 Denote by $R_{m,\lambda}$ the $T_\CC$-invariant subspace consisting of all the sections of $H^0(X,mL)$ with weight $\lambda\in P_m$. So we have
\begin{equation*}
	H^0(X,mL)=\bigoplus_{\lambda\in P_m}R_{m,\lambda}.
\end{equation*}
Recall $d_m:=\dim H^0(X,mL)$ and set
\begin{equation*}
		d_{m,\lambda}:=\dim R_{m,\lambda}, \quad
		\overline{g_m}:=
\frac1{d_m}\sum_{\lambda\in P_m}g(\lambda/m)d_{m,\lambda}
\end{equation*}

Now we further assume that $mL$ is very ample. Then one can consider the $T$-invariant subspace of $\cB_m$, denoted $\cB_m^T$. Any element in $\cB^T_m$ takes the form
$$
\varphi=\frac{1}{m}\log\sum_{\lambda\in P_m}\sum_{\alpha=1}^{d_{m,\lambda}}|s^{(\lambda)}_\alpha|^2_{h^m},
$$
where $\{s^{(\lambda)}_\alpha\}_{1\le\alpha\le d_{m,\lambda}}$ is any basis of $R_{m,\lambda}$. 
Moreover, any Bergman geodesic in $\cB^T_m$ takes the form
$$
\varphi(t)=\frac{1}{m}\log\sum_{\lambda\in P_m}\sum_{\alpha=1}^{d_{m,\lambda}}e^{\gamma^{(\lambda)}_\alpha t}|s^{(\lambda)}_\alpha|^2_{h^m}
$$
for some $\gamma^{(\lambda)}_\alpha\in\RR$. Also note that $\cB^T_m$ inherits the $d_{1,m}$-distance from the ambient space $\cB_m$.

Consider
$$
H_m:=\int_Xh^m(\,\cdot\,,\,\cdot\,)\omega^n,
$$
which is a $T$-invariant Hermitian inner product on $H^0(X,mL)$. So in particular, we have
\begin{equation}
\label{eq:T-inv-orthogonality}
    H_m(s^{(\lambda_1)},s^{(\lambda_2)})=0,\ \text{for }s^{(\lambda_i)}\in E_{m,\lambda_i}\ \text{whenever }\lambda_1\neq\lambda_2.
\end{equation}
(Such orthogonality holds for any $T$-invariant $H\in\cP_m$.)
We put, as in \eqref{eq:def-E-m},
$$
E_m(\varphi):=E_m(H_m,\varphi),\ \varphi\in\cB_m^T.
$$
The only difference here is that everything is $T$-invariant.

Following \cite{BWN14}, we consider the $g$-weighted $\log\det$ functional:
\begin{equation}
\label{eq:def-E-g-m}
	E^g_m(\varphi):=\frac{1}{md_m\overline{g_m}}\sum_{\lambda\in P_m}g(\lambda/m)\log\det\bigg(\int_Xh^m(s^{(\lambda)}_\alpha,s^{(\lambda)}_\beta)\omega^n\bigg)
\end{equation}
for any $\varphi=\frac{1}{m}\log\sum_{\lambda\in P_m}\sum_{\alpha=1}^{d_{m,\lambda}}|s^{(\lambda)}_\alpha|^2_{h^m}\in\cB^T_m$. As the notation suggests, $E_m^g$ is the quntization of $E^g$ (cf. \cite{BWN14,Lah19}).

The following estimate follows easily from Lemma \ref{lem:sup=max-mu} and Lemma \ref{lem:mu-max=sup-E}.
\begin{lemma}
\label{lem:sup-E-g-m<>sup-E-m}
	There exists $\varepsilon_m\rightarrow0$ such that
	$$
	\frac{\inf_Pg}{\sup_Pg}(\sup\varphi-E_m(\varphi))-\varepsilon_m\le(\sup\varphi-E^g_m(\varphi))\le\frac{\sup_Pg}{\inf_Pg}(\sup\varphi-E_m(\varphi))+\varepsilon_m
	$$
	for any $\varphi\in\cB^T_m$.
\end{lemma}

Now mimicking the definition of $\delta^A_m(L)$, we introduce the following
\begin{definition}
Set
	$$
	\delta^{A,g}_m(L)=\sup\bigg\{\delta>0\,:\,\sup_{\varphi\in\cB_m^T}\int_Xe^{-\delta(\varphi-E^g_m(\varphi))}\omega^n<\infty\bigg\}.
	$$
\end{definition}
In the case when $mL$ is not very ample, one can still make sense of the above definition by considering instead the following integral
\begin{equation}
\label{eq:I=int-mu-lambda-alpha}
	I:=\int_X\frac{\prod_{\lambda\in P_m}\prod_{\alpha=1}^{d_{m,\lambda}}\bigg(\mu^{(\lambda)}_\alpha\bigg)^{\frac{\delta g(\lambda/m)}{md_m\overline{g_m}}}}{\bigg(\sum_{\lambda\in P_m}\sum_{\alpha=1}^{d_{m,\lambda}}\mu_\alpha^{(\lambda)}|s^{(\lambda)}_\alpha|^2_{h^m} \bigg)^{\frac{\delta}{m}}}\omega^n,
\end{equation}
where $\{s^{(\lambda)}_\alpha\}$ is any $H_m$-orthonormal basis and $\mu^{(\lambda)}_\alpha>0$ are some parameters.
Then $\delta^{A,g}_m(L)$ is the supremum of all the $\delta>0$ such that $I$ is uniformly bounded from above by some $C_\delta>0$ for any choice of 
$\{\mu^{(\lambda)}_\alpha, s^{(\lambda)}_\alpha\}$.
Notice that  applying Young's inequality (compare Corollary \ref{cor:delta-m<=delta-m-A}), we deduce that
\begin{equation}
\label{eq:Young}
I\le\prod_{\lambda\in P_m}\bigg(\frac{g(\lambda/m)}{d_m\overline{g_m}}\bigg)^{\frac{\delta g(\lambda/m)d_{m,\lambda}}{md_m\overline{g_m}}}\int_X\frac{\omega^n}{\prod_{\lambda\in P_m}\prod_{\alpha=1}^{d_{m,\lambda}}\big|s^{(\lambda)}_\alpha\big|_{h^m}^{\frac{2\delta g(\lambda/m)}{md_m\overline{g_m}}}}
\end{equation}

\subsection{Algebraic $g$-weighted $\delta_m$-invariant}
\label{sec:def-delta-g-m}
Inspired by the above inequality, we can now introduce the algebraic $g$-weighted $\delta_m$-invariant as follows.
Let $m\ge1$ be any integer such that $H^0(X,mL)\neq\{0\}$.
For each weight space $R_{m,\lambda}$, we pick a basis, say $\{s_\alpha^{(\lambda)}\}_{1\le\alpha\le d_{m,\lambda}}$. Put
\begin{equation}
\label{eq:def-m-g-basis-divisor}
	D:=\frac{1}{md_m\overline{g_m}}\sum_{\lambda\in P_m}\sum_{\alpha=1}^{d_{m,\lambda}}g(\lambda/m)\{s^{(\lambda)}_\alpha=0\}.
\end{equation}
Then $D$ is an effective $\RR$-divisor that is $\RR$-linear equivalent to $L$. Such a divisor will be called an $(m,g)$-basis divisor. 
\begin{definition}
\label{def:delta-g-m}
Set
\begin{equation*}
	\delta^g_m(L):=\inf\bigg\{\lct(X,D)\,:\,D\text{ is an }(m,g)\text{-basis divisor} \bigg\},
\end{equation*}
and
\begin{equation*}
	\delta^g(L):=\limsup_{m\rightarrow\infty}\delta^g_m(L).
\end{equation*}
\end{definition}
Observe that  any $(m,g)$-basis divisor is $T_\CC$-invariant. Then its lct can be computed by some $T_\CC$-invariant divisor over $X$ (by applying a $T_\CC$-equivariant log resolution).
 So in what follows we will only consider $T_\CC$-invariant divisors over $X$.

Let $F$ be a $T_\CC$-invariant prime divisor over $X$. For any subspace $V\subseteq H^0(X,mL)$, put
\begin{equation*}
			\cF^a_{\ord_F}V:=\{s\in V\,:\,\ord_F(s)\ge a\}.\\
\end{equation*}
Then we define the \emph{$g$-weighted \mthit\ expected vanishing order} $S^g_m(F)$ to be
\begin{equation}
\label{eq:def-S-m-g}
	S^g_m(F):=S^g_m(L;F):=\frac{1}{md_m\overline{g_m}}\sum_{\lambda\in P_m}\sum_{a\ge1}g(\lambda/m)\dim\cF^a_{\ord_F}R_{m,\lambda}.
\end{equation}

\begin{lemma}
\label{lem:delta-g-m=inf-A/S}
One has
	$$
	\delta^g_m(L)=\inf_F\frac{A_X(F)}{S^g_m(F)},
	$$
	where $F$ runs through all the $T_\CC$-invariant prime divisors over $X$.
\end{lemma}

\begin{proof}
The result follows by applying the argument of \cite[Lemma 2.2]{FO18} to each weight space $R_{m,\lambda}$.        
\end{proof}

\begin{lemma}
\label{lem:finite-lct-delta-g-m}
$\delta^g_m(L)$ is computed by some $T_\CC$-invariant $F$ over $X$.
\end{lemma}

\begin{proof}
	For each $m$, the coefficients of all the $(m,g)$-basis divisors are contained in a finite set. In other words, these divisors lie in a finite combination of linear systems. Even though the coefficients in front of each linear system is possibly irrational, one can still apply \cite{Am16} to conclude that  the lct of all the $(m,g)$-basis divisors can only take finitely many values. So one can find some $F$ computing $\delta_m^g(L)$.
\end{proof}

We have the following $g$-weighted version of Theorem \ref{thm:delta-A-m=delta-m}.
\begin{theorem}
\label{thm:delta-g-m=delta-g-A-m}
	One has
	$
	\delta^{A,g}_m(L)=\delta^g_m(L).
	$
\end{theorem}

\begin{proof}
We only give an outline.
Firstly, following the proof of Proposition \ref{prop:deltam=int<C}, one sees that  to compute $\delta^g_m$-invariant, it is enough to consider $H_m$-orthonormal basis $\{s^{(\lambda)}_\alpha\}$.
	Then the direction $\delta^{A,g}_m(L)\ge\delta^g_m(L)$ follows from \eqref{eq:Young}. To show the reverse direction, we consider any $T_\CC$-invariant prime divisor $F$ over $X$. Then we can pick an $H_m$-orthonormal basis, say $\{s^{(\lambda)}_\alpha\}$, which is compatible with the filtration $\cF^\bullet_{\ord_F}$ on each $E_{\lambda,m}$. Then for any parameter $t\ge0$, put
	$
	\mu^{(\lambda)}_\alpha:=e^{t\,\ord_F(s^{(\lambda)}_\alpha)}.
	$
	Plugging these into the integral \eqref{eq:I=int-mu-lambda-alpha}, we find that
	$$
	I=e^{t\big(\delta S^g_m(F)-A_X(F)\big)}\int_X\frac{e^{tA_X(F)}\omega^n}{\bigg(\sum_{\lambda\in P_m}\sum_{\alpha=1}^{d_{m,\lambda}}e^{t\,\ord_F(s^{(\lambda)}_\alpha)}|s^{(\lambda)}_\alpha|^2_{h^m}\bigg)^{\frac{\delta}{m}}}.
	$$
	Then $\delta^{A,g}_m(L)\le\delta^g_m(L)$ follows from Lemma \ref{lem:I-lambda>eta} and \ref{lem:delta-g-m=inf-A/S}.
	\end{proof}
\subsection{$g$-weighted $(f,\delta)$-balanced metrics}
Let $m$ be sufficiently divisible and assume that $mL$ is very ample. For any $T$-invariant $f\in C^\infty(X,\RR)$ and $\delta>0$, we put
\begin{equation*}
	F^{f,\delta,g}_m(\varphi):=-\frac{1}{\delta}\log\frac{1}{V}\int_Xe^{f-\delta\varphi}\omega^n-E_m^g(\varphi),\ \varphi\in\cB^T_m.
\end{equation*}

\begin{definition}
	Any critical point $\varphi$ of $F_m^{f,\delta,g}$ is called $g$-weighted $(f,\delta)$-balanced, and $\omega_\varphi:=\omega+\ddc\varphi$ is called an $g$-weighted $(f,\delta)$-balanced metric.
\end{definition}

A
computation shows that if
$
\varphi=\frac{1}{m}\log\sum_{\lambda\in P_m}\sum_{\alpha=1}^{d_{m,\lambda}}|s^{(\lambda)}_\alpha|^2_{h^m}\in\cB^T_m
$
is a critical point of $F^{f,\delta,g}_m$, then
\begin{equation}
\label{eq:g-balanced-equation}
	\frac{d_m\overline{g_m}}{g(\lambda/m)\int_Xe^{f-\delta\varphi}\omega^n}\int_X(he^{-\varphi})^m(s^{(\lambda)}_\alpha,s^{(\lambda)}_\beta)e^{f-\delta\varphi}\omega^n=\delta_{\alpha\beta}
\end{equation}
for any $\lambda\in P_m$, $1\le\alpha,\beta\le d_{m,\lambda}$.

\begin{definition}
	We say $F_m^{f,\delta,g}$ is coercive on $\cB^T_m$ if there exist $\gamma>0$ and $C>0$ such that
	$$
	F^{f,\delta,g}_m(\varphi)\ge\gamma(\sup\varphi-E^g_m(\varphi))-C,\ \forall\varphi\in\cB^T_m.
	$$
\end{definition}

\begin{proposition}
\label{prop:F-f-delta-g-coerc=>g-balanced}
	If $F_m^{f,\delta,g}$ is coercive on $\cB^T_m$, then there exists a $g$-weighted $(f,\delta)$-balanced $\varphi\in\cB^T_m$, minimizing $F_m^{f,\delta,g}$.
\end{proposition}

\begin{proof}
	It follows from Lemma \ref{lem:sup-E-g-m<>sup-E-m} and the proof of Proposition \ref{prop:F-m-coercive=>balanced}.
\end{proof}

The next result is a $g$-weighted version of Proposition \ref{prop:F-coerc=delta-m}.
\begin{proposition}
\label{prop:F-delta-g-m-coerc=delta-g-m}
	$F_m^{f,\delta,g}$ is coercive on $\cB^T_m$ if and only if $0<\delta<\delta^g_m(L)$.
\end{proposition}

\begin{proof}
	The proof is the almost identical to the one for Proposition \ref{prop:F-coerc=delta-m}. The only thing we need to check is the following inequality:
	$$
	S^g_m(F_0)<T_m(F_0),
	$$
	where $F_0$ is the $T_\CC$-invariant divisor computing $\delta^g_m(L)$ (Lemma \ref{lem:finite-lct-delta-g-m}).
	Suppose otherwise that we have the equality. Then we must have
	$$
	\dim\cF^a_{\ord_{F_0}}R_{m,\lambda}=\dim R_{m,\lambda},\ \forall\lambda\in P_m,\ \forall 1\le a\le \tau_m(F_0).
	$$
	This implies that
	$$
	\dim\cF^a_{\ord_{F_0}}H^0(X,mL)=\sum_{\lambda\in P_m}\dim R_{m,\lambda}=\dim H^0(X,mL),\ \forall 1\le a\le \tau_m(F_0),
	$$
	contradicting the base point freeness of $|mL|$.
\end{proof}

\subsection{$\theta$-twisted $g$-balanced metric}
We fix a $T$-invariant smooth form $\theta$ in $c_1(X)-c_1(L)$. Let $f_\theta$ be the $T$-invariant function defined by \eqref{eq:def-f-alpha}. 
Consider the following $g$-weighted quantized $\theta$-Ding functional:
\begin{equation*}
	F^{\theta,g}_m(\varphi):=-\log\frac{1}{V}\int_Xe^{f_\theta-\varphi}\omega^n-E_m^g(\varphi),\ \varphi\in\cB_m^T.
\end{equation*}

Next we introduce a natural quantization of the $\theta$-twisted K\"ahler--Ricci $g$-soliton.
\begin{definition}
	Any critical point $\varphi\in\cB_m$ of $F^{\theta,g}_m$ is 
called $(\theta,g)$-balanced, and $\omega_\varphi:=\omega+\ddc\varphi$ is called an $(\theta,g)$-balanced metric.
\end{definition}

Then we can state the $g$-weighted version of Theorem \ref{thm:finite-dim-YTD}.
\begin{theorem}
\label{thm:finite-dim-g-weighted-YTD}
	The following statements hold:	\begin{enumerate}
		\item If $\delta^g_m(L)>1$, then there exists an $(\theta,g)$-balanced $\varphi$ in $\cB^T_m$.
		\item Assume further that $\alpha\ge0$. If there exists an $(\theta,g)$-balanced $\varphi$ (resp. a unique $(\theta,g)$-balanced $\varphi$ up to constant) in $\cB^T_m$, then $\delta^g_m(L)\ge1$ (resp. $\delta^g_m(L)>1$).
	\end{enumerate}
\end{theorem}
\begin{proof}
	It is enough to notice that  when $\theta\ge0$, one can still apply \cite{Bern15} to get the convexity of $F^{\theta,g}_m$ along Bergman geodesics in $\cB^T_m$. Then, after restricting everything to the $T$-invariant subspace $\cB^T_m$, the proof is almost the same as the one for Theorem \ref{thm:finite-dim-YTD}.
\end{proof}

In the $g$-weighted setting, one can also prove analogous approximation results following the lines in \S\ref{sec:limit-behavior}. We leave this to the interested readers.

\subsection{Further properties}
In this part we collect some properties of the $\delta^g$-invariant.

Following \cite[(5.41)]{HL20}, for any $T_\CC$-invariant $F$ over $X$ and $t\geq0$, put
$$
\vol^g(L-tF):=\lim_{m\rightarrow\infty}\frac{n!}{m^n}\sum_{\lambda\in P_m}g(\lambda/m)\dim \cF^{mt}_{\ord_F}R_{m,\lambda}.
$$
Thus we find that
$$
S^g(F):=S^g(L;F):=\lim_{m\rightarrow\infty}S^g_m(F)=\int_0^{\infty}\vol^g(L-xF)dx.
$$

Then choosing a Newton--Okounkov body that respects the $T_\CC$-action (cf. \cite[Section 3.2.2]{LX17}) and extending \cite[Section 2]{BJ17} to the weighted setting, one obtains:
\begin{proposition}
\lb{deltagprop}
The limsup in Definition \ref{def:delta-g-m} is in fact a limit, and we have
	$$
	\delta^g(L)=\lim_{m\rightarrow\infty}\delta^g_m(L)=\inf_F\frac{A_X(F)}{S^g(F)},
	$$
	where $F$ runs through all the $T_\CC$-invariant prime divisors over $X$.
\end{proposition}
It is interesting to compare Proposition \ref{deltagprop} with the
recent work \cite[Section 5.6]{HL20} that relied on the non-Archimedean approach.

Finally, we state a result involving the \emph{the greatest Bakry--Emery Ricci lower bound}. Put
\begin{equation*}
	\beta^g(L):=\sup\bigg\{\beta\in\RR\,:\,\exists T\text{-invariant }\omega
	\text{ with} [\omega]=c_1(L)\text{ s.t. } \Ric(\omega)\ge\beta\omega+\ddc\log g(m_\omega)\bigg\}.
\end{equation*}
Then following \cite{BBJ18,CRZ19,HL20},\cite[Corollary 3.10]{Zha20}:
\begin{proposition}
\label{prop:beta-g=delta-g}
	$\beta^g(L)=\min\{s(L),\delta^g(L)\}=\min\{s(L),\delta^{A,g}(L)\}.$
\end{proposition}

\section{$\delta_m$-invariants associated to torus actions
}
\label{sec:delta-T}

Let $T=(S^1)^r$ and denote by $T_\CC\cong(\CC^*)^r$ its complexification.
Suppose that $(X,L)$ admits a holomorphic $T_\CC$-action (extending to the
total space of $L$ and preserving fibers).
Our goal in this part is to study the $T_\CC$-equivariant $\delta_m$
defined by
\begin{equation}
\label{eq:def-delta-T-m}
	\delta^{T_\CC}_m(L):=\inf_{F\text{}\ T_\CC\text{-invariant}}
\frac{A_X(F)}{S_m(F)},
\end{equation}
where a divisor $F$ over $X$ is $T_\CC$-invariant if
$
\ord_F(\tau\cdot s)=\ord_F(s)
$
 for any 
$s\in H^0(X,mL), \,\tau\in T_\CC$.
Somewhat surprisingly, we show that the $g$-weighted analysis of \S\ref{sec:soliton} 
has some important consequences already when $g\equiv1$. In fact, in that case, 
$\delta^g_m(L)$ coincides with $\delta^{T_\CC}_m(L)$, and Theorem
\ref{thm:delta-T-m}, that we now prove, relates $\delta^{T_\CC}_m(L)$ to $\delta_m(L)$.

\begin{proof}[Proof of Theorem \ref{thm:delta-T-m}]
	We assume that $s(L)>0$, otherwise the statement is trivial. It suffices to show that  for any $\delta\in(0,\min\{\delta_m^{T_\CC}(L),s(L)\})$,
	$$
	\delta_m(L)\ge\delta.
	$$
	To this end, pick a $T$-invariant smooth form $\theta\in(c_1(X)-\delta c_1(L))$. We may assume $\theta$ to be semi-positive as $\delta<s(L)$. 
	Let $f_\theta$ be any $T$-invariant smooth function satisfying
	$$
	\Ric(\omega)=\delta\omega+\theta+\ddc f_\theta.
	$$
	Consider
	$$
	F^{f_\theta,\delta,g}_m(\varphi):=-\frac{1}{\delta}\log\int_Xe^{f_\theta-\delta\varphi}\omega^n-E_m^g(\varphi),\ \varphi\in\cB^T_m.
	$$
	When $g\equiv1$, one simply has (recall Definition \ref{def:F-f-delta-m}, \eqref{eq:def-E-m} and \eqref{eq:def-E-g-m})
	$$
	F_m^{f_\theta,\delta,g}(\varphi)=F^{f_\theta,\delta}_m\big|_{\cB^T_m}(\varphi)=-\frac{1}{\delta}\log\int_Xe^{f_\theta-\delta\varphi}\omega^n-E_m(\varphi),\ \varphi\in\cB^T_m.
	$$
	Since $\delta^{T_\CC}_m(L)>\delta$, Propositions \ref{prop:F-delta-g-m-coerc=delta-g-m} and \ref{prop:F-f-delta-g-coerc=>g-balanced} then imply that there exists a $1$-weighted $(f_\theta,\delta)$-balanced $\varphi\in\cB^T_m$ minimizing $F^{f_\theta,\delta}_m\big|_{\cB^T_m}$. Note that this $\varphi$ is nothing but a $(f_\theta,\delta)$-balanced potential in the sense of Definition \ref{def:F-f-delta-m}, only with the additional property that it is $T$-invariant (compare \eqref{eq:balanced-equation} and \eqref{eq:g-balanced-equation}; recall also \eqref{eq:T-inv-orthogonality}). So it is also a critical point of $F^{f_\theta,\delta}_m$ on the whole space $\cB_m$. Now using the non-negativity of $\theta$, we see that $F^{f_\theta,\delta}_m$ is convex along Bergman geodesics in $\cB_m$ (by \cite{Bern15} again). This shows that $\varphi$ is a minimizer of $F^{f_\theta,\delta}_m$, and hence
	$$
	F^{f_\theta,\delta}_m\ge-C \text{ on }\cB_m
	$$
	for some $C>0$.
	Then by Theorem \ref{thm:delta-A-m=delta-m},
	$\delta_m(L)=\delta^A_m(L)\ge\delta$, as desired.
\end{proof}

Applying the above result to the toric Fano setting can give us the precise formula of $\delta_m(-K_X)$ for any $m\ge1$. More precisely, following the terminologies in Fulton \cite{Ful93}, let $X$ be a toric Fano manifiold with an associated fan $\Delta$ in a lattice $N\cong\ZZ^n$. Let $M$ be the dual lattice of $N$. Given any ray $\rho\in\Delta$, its primitive generator is denoted by 
$
v_\rho,
$
whose corresponding toric divisor is denoted by $D_\rho$. Then
$
-K_X=\sum_\rho D_\rho.
$
We take $T_\CC\cong(\CC^*)^n$ to be the maximal torus, then the moment polytope $P$ is given by
$$
P:=\{u\in M_\RR\ |\ \langle u,v_\rho\rangle+1\ge0,\ \forall\rho\}.
$$
In this case, each weight space $R_{m,\lambda}$ is one-dimensional. So the $(m,1)$-basis divisor is uniquely determined (recall \eqref{eq:def-m-g-basis-divisor}), from which we deduce that
$$
\delta^{T_\CC}_m(-K_X)=\min_{\rho}\bigg\{\frac{1}{\langle b_m,v_\rho\rangle+1}\bigg\},
$$
where $b_m$ is the quantized barycenter given by
$
b_m:=\frac{1}{md_m}\sum_{u\in mP\cap M}u.
$
By completeness of the fan $\Delta$, there always exists a ray $\rho\in\Delta$ with $\langle b_m,v_\rho\rangle\geq0$. So we have
$
\delta^{T_\CC}_m(-K_X)\le s(-K_X)=1.
$
Thus Theorem \ref{thm:delta-T-m} implies the following

\begin{corollary}
\label{cor:delta-m-toric-Fano}
For toric Fano manifold $X$ and $m\ge1$, we have
$$
\delta_m(-K_X)=\delta^{T_\CC}_m(-K_X)=\min_{\rho}\bigg\{\frac{1}{\langle b_m,v_\rho\rangle+1}\bigg\}.
$$
\end{corollary}

Finally we remark that, even when $X=\PP^n$, this gives a new result.

\begin{corollary}
Let $X=\PP^n$. Then for any $m\ge1$, we have
	$
	\delta_m(-K_X)=1.
	$
\end{corollary}

\appendix

\section{Extension to coupled metrics}
\label{sec:coupled-soliton}
In this part we define a coupled $\delta$-invariant for the coupled system of Monge--Amp\`ere equations studied in \cite{HNW19,Hul19,DH18} and record
analogues of our main theorems to this more general setting.

\subsection{Coupled K\"ahler--Ricci $\boldsymbol{g}$-soliton}
We follow the setup of \cite{DH18}.
Let $X$ be a projective manifold. Fix a positive integer $k$ and take a $k$-tuple of ample $\QQ$-line bundles 
$(L_1,\cdots,L_k).$ 
As in the previous section, assume that there is an effective and holomorphic $T_\CC$-action on $X$. Also assume that this torus action lifts to each $L_i$. We equip each $L_i$ with a positively curved smooth $T$-invariant Hermitian metric $h_i$, whose curvature form will be denoted by $\omega_i$. 
Put
$
    V_i:=\int_X\omega_i^n,\ 1\leq i\leq k.
$
Denote by
$
\cH^T_\omega(X,\omega_i)
$
the subspace of $\cH(X,\omega_i)$ consisting of $T$-
invariant K\"ahler potentials and put
\begin{equation*}
\mathbfcal{H}^T:=\cH^T_\omega(X,\omega_1)\times\cdots\times\cH^T_\omega(X,\omega_k).
\end{equation*}
Note that each $\omega_i$ induces a moment map
$
m_{\omega_i}:X\rightarrow\RR^r,
$
whose image will be denoted by $P_i$. Recall that $P_i$ is a polytope, which does not depend on the choice of $\omega_i\in c_1(L_i)$. We will fix a smooth positive function
$
	g_i:P_i\rightarrow\RR_{>0}
$
for each $1\leq i\leq k$. Then for any $\varphi_i\in\cH^T_\omega(X,\omega_i)$, we have an induced function
$
	g_i(\varphi_i):=g_i\circ m_{\omega_i+\ddc\varphi_i}:X\rightarrow\RR_{>0}.
$
Put
\begin{equation*}
	\boldsymbol{g}:=(g_1,\cdots,g_k).
\end{equation*}
To set up the coupled soliton equations, we need also a twist term, by cohomological reason. Pick a $T$-invariant smooth form
$
	\theta
$
in
$
c_1(X)-\sum_{i=1}^k c_1(L).
$
Given a $k$-tuple $\boldsymbol{\varphi}:=
(\varphi_1,\cdots,\varphi_k)\in\mathbfcal{H}^T$,
we put
$
	\boldsymbol{\omega_\varphi}:=(\omega_1+\ddc\varphi_1,\cdots,\omega_i+\ddc\varphi_k).
$
\begin{definition}\cite{HNW19,Hul19,DH18}
We say that $\boldsymbol{\omega_\varphi}$ is an $\theta$-twisted coupled K\"ahler--Ricci $\boldsymbol{g}$-soliton if
$$
\Ric(\omega_i+\ddc\varphi_i)=\sum_{j=1}^k(\omega_j+\ddc\varphi_j)+\theta+\ddc\log g_i(\varphi_i),\ 1\leq i\leq k.
$$
\end{definition}

To study the above coupled soliton equations, one needs a coupled Ding function, which we now describe. For each $1\leq i\leq k$, one can find $f_i\in C^\infty(X,\RR)$ satisfying
$\Ric(\omega)_i=\sum_{i=1}^k\omega_i+\theta+\ddc f_i$ and  $\int_Xe^{f_i}\omega_i^n=V_i.
$
Then it is easy to see that  as probability measures,
$
	\frac{e^{f_1}\omega_1^n}{V_1}=\cdots=\frac{e^{f_k}\omega_k^n}{V_k}.
$
So we can put
\begin{equation*}
	\mu:=\frac{e^{f_1}\omega_1^n}{V_1}=\cdots=\frac{e^{f_k}\omega_k^n}{V_k}.
\end{equation*}
Note that $\mu$ depends on $\omega_i$ and $\alpha$.
Now, following \cite{DH18}, the $\boldsymbol{g}$-weighted $\theta$-Ding functional is defined by
\begin{equation*}
	D^{\alpha,\boldsymbol{g}}(\boldsymbol{\varphi} ):=-\log\int_X e^{-\sum_i\varphi_i}d\mu-\sum_{i=1}^{k}E^{g_i}_{\omega_i}(\varphi_i),\ \boldsymbol{\varphi}:=
(\varphi_1,\cdots,\varphi_k)\in\mathbfcal{H}^T,
\end{equation*}
where
$
	E^{g_i}_{\omega_i}(\varphi_i)
$
is the $g_i$-weighted Monge-Amp\`ere energy (recall \eqref{eq:def-E-g}). Then it is straightforward to check that  $\boldsymbol{\omega_\varphi}$ is $\theta$-twisted coupled K\"ahler--Ricci $\boldsymbol{g}$-soliton if and only if $\boldsymbol{\varphi}$ a critical point of $D^{\alpha,\boldsymbol{g}}$.

As shown in \cite{DH18}, there are obstructions to the existence of K\"ahler--Ricci $\boldsymbol{g}$-solitons, mainly coming from the non-coercivity of $D^{\alpha,\boldsymbol{g}}$. So we introduce a coupled coercivity threshold as follows:
	$$
		\delta^{A,\boldsymbol{g}}(L_1,\cdots,L_k):=\sup\bigg\{\delta>0\,:\,\sup_{(\varphi_1,\cdots,\varphi_k)\in\mathbfcal{H}^T}\int_Xe^{-\delta\sum_{i=1}^k(\varphi_i-E^{g_i}_{\omega_i}(\varphi_i))}d\mu<\infty
\bigg\}.
	$$
\subsection{Quantization}Following \cite{BWN14,Tak19}, we can quantize the above setup as follows. Choose a $k$-tuple of sufficiently divisible integers 
\begin{equation*}
    \boldsymbol{m}:=(m_1,\cdots,m_k)
\end{equation*}
 such that each $m_iL_i$ is very ample. Consider the Bergman space $\cB_{m_i}(X,\omega_i)$ for each pair $(L_i,\omega_i)$. We denote by $\cB_{m_i}^T(X,\omega_i)$ the subspace of $T$-invariant Bergman potentials in $\cB_{m_i}(X,\omega_i)$ and put
\begin{equation*}
	\boldsymbol{\cB}_{\boldsymbol{m}}^T:=\cB_{m_1}^T(X,\omega_1)\times\cdots\times\cB_{m_k}^T(X,\omega_k).
\end{equation*}

Note that  the $T_\CC$-action induces a weight decomposition for each vector space $H^0(X,m_iL_i)$. 
Then proceeding as in  \S\ref{sec:quantize-g-soliton} (especially see \eqref{eq:def-E-g-m}),
$
	E^{g_i}_{\omega_i,m_i}(\cdot) \text{ on } \cB^T_{m_i}(X,\omega_i),
$
the $m_i$-th quantized $g_i$-weighted Monge--Amp\`ere energy of the pair $(L_i,\omega_i)$.
Then we put
	$$
		\delta_{\boldsymbol{m}}^{A,\boldsymbol{g}}(L_1,\cdots,L_k):=\sup\bigg\{\delta>0\,:\,\sup_{(\varphi_1,\cdots,\varphi_k)\in\mathbfcal{B}_{\boldsymbol{m}}^T}\int_Xe^{-\delta\sum_{i=1}^k(\varphi_i-E^{g_i}_{\omega_i,m_i}(\varphi_i))}d\mu<\infty
\bigg\}.
	$$

\subsection{Algebraic coupled $\boldsymbol{g}$-weighted $\delta_{\boldsymbol{m}}$-invariant}Motivated by the above formulation, we can now define the algebraic coupled $\boldsymbol{g}$-weighted $\delta_{\boldsymbol{m}}$ in the following way.

 Following Section \ref{sec:def-delta-g-m}, choose an $(m_i,g_i)$-basis divisor
$
D_i\sim_\RR L_i
$
for each $L_i$.
Summing up, we get a $T_\CC$-invariant effective $\RR$-divisor
$
D:=\sum_{i=1}^kD_i\sim_\RR\sum_{i=1}^kL_i,
$
which will be called an $(\boldsymbol{m,g})$-basis divisor of the $k$-tuple $(L_1,\cdots,L_k)$.

\begin{definition} 
\label{def:coupled-delta-m}
\begin{enumerate}
\item The coupled $\boldsymbol{g}$-weighted $\delta_{\boldsymbol{m}}$-invariant is
\begin{equation*}
	\delta^{\boldsymbol{g}}_{\boldsymbol{m}}(L_1,\cdots,L_k):=\inf\bigg\{\lct(X,D)\,:\,D\text{ is an }(\boldsymbol{m,g})\text{-basis divisor of } (L_1,...,L_k)\bigg\}.
	\end{equation*}
	
\item 	When $\boldsymbol{g}=(1,\cdots,1)$, set
$\delta^{T_\CC}_{\boldsymbol{m}}(L_1,\cdots,L_k):=\delta^{\boldsymbol{g}}_{\boldsymbol{m}}(L_1,\cdots,L_k).
$

\item 	When the torus action is trivial, set
	$$
	\delta_{\boldsymbol{m}}(L_1,\cdots,L_k):=\inf\bigg\{\lct\Big(X,\sum_{i=1}^{k}D_i\Big)\,:\,\text{ each }D_i\text{ is an $m_i$-basis divisor of }L_i\bigg\}.
	$$

\end{enumerate}

\end{definition}

Since any $(\boldsymbol{m,g})$-basis divisor is $T_\CC$-invariant, to compute its lct, it suffices to investigate all the $T_\CC$-invariant prime divisors $F$ over $X$. Then (recall \eqref{eq:def-S-m-g}) one can consider
$
	S_{m_i}^{g_i}(L_i;F),
$
the $m_i$-th $g_i$-weighted expected vanishing order of $L_i$ along $F$.

\begin{lemma}
\label{lem:coupled-delta-g-m=inf-A/S}
	One has
	$
	\delta^{\boldsymbol{g}}_{\boldsymbol{m}}(L_1,\cdots,L_k)=\inf_F\frac{A_X(F)}{\sum_{i=1}^kS^{g_i}_{m_i}(L_i;F)},
	$
	where $F$ runs through all the $T_\CC$-invariant prime divisors over $X$.
\end{lemma}

And also, we have the following result.

\begin{lemma}
	$\delta^{\boldsymbol{g}}_{\boldsymbol{m}}$ is computed by some $T_\CC$-invariant divisor $F$ over $X$.
\end{lemma}

The coupled version of Theorem \ref{thm:delta-g-m=delta-g-A-m} also holds.

\begin{theorem}
	One has
	$
	\delta^{A,\boldsymbol{g}}_{\boldsymbol{m}}(L_1,\cdots,L_k)=\delta^{\boldsymbol{g}}_{\boldsymbol{m}}(L_1,\cdots,L_k).
	$
\end{theorem}

\subsection{Coupled $\boldsymbol{g}$-weighted $(f,\delta)$-balanced metrics}
For any $\delta>0$ and $f\in C^\infty(X,\RR)$, put
\begin{equation*}
	F_{\boldsymbol{m}}^{f,\delta,\boldsymbol{g}}(\boldsymbol{\varphi}):=-\frac{1}{\delta}\log\int_Xe^{f-\delta\sum_{i=1}^k\varphi_i}d\mu-\sum_{i=1}^{k}E^{g_i}_{\omega_i,m_i}(\varphi_i),\ \boldsymbol{\varphi}=(\varphi_1,...,\varphi_k)\in\boldsymbol{\cB}^T_{
	\boldsymbol{m}}.
\end{equation*}

\begin{definition}
	Any critical point of $F_{\boldsymbol{m}}^{f,\delta,\boldsymbol{g}}$ is called coupled $\boldsymbol{g}$-weighted $(f,\delta)$-balanced.
\end{definition}

\begin{remark}
{\rm	One can think of $F_{\boldsymbol{m}}^{f,\delta,\boldsymbol{g}}$ as the quantization of 
	$$
	F_{\boldsymbol{m}}^{f,\delta,\boldsymbol{g}}(\boldsymbol{\varphi}):=-\frac{1}{\delta}\log\int_Xe^{f-\delta\sum_{i=1}^k\varphi_i}d\mu-\sum_{i=1}^{k}E^{g_i}_{\omega_i}(\varphi_i),\ \boldsymbol{\varphi}=(\varphi_1,...,\varphi_k)\in\boldsymbol{\cH}^T,
	$$
	with 
$		
\Ric(\omega_i+\ddc\varphi_i)
=\delta
\sum_{j=1}^{k}
(\omega_j+\ddc\varphi_j)
+(1-\delta)\sum_{j=1}^k\omega_j
+(\alpha-\ddc f)+\ddc \log g_i(\varphi_i),\ 1\leq i\leq k,
$
being the
critical point equations.
	Varying $\delta$, this can be seen as a continuity path towards the $(\theta-\ddc f)$-twisted coupled K\"ahler--Ricci $\boldsymbol{g}$-soliton metric (cf. \cite{DH18}).
	}
\end{remark}

\begin{definition}
	We say $F_{\boldsymbol{m}}^{f,\delta,\boldsymbol{g}}$ is coercive on $\boldsymbol{\cB}^T_{\boldsymbol{m}}$ if there exist $\lambda>0$ and $C>0$ such that
	$$
	F_{\boldsymbol{m}}^{f,\delta,\boldsymbol{g}}(\boldsymbol{\varphi})\geq\lambda\sum_{i=1}^k(\sup\varphi_i-E^{g_i}_{\omega_i,m_i}(\varphi_i))-C,\ \forall\boldsymbol{\varphi}=(\varphi_1,...,\varphi_k)\in\boldsymbol{\cB}^T_{
	\boldsymbol{m}}.
	$$
\end{definition}

The next result follows from the argument for Propositions \ref{prop:F-m-coercive=>balanced} and \ref{prop:F-f-delta-g-coerc=>g-balanced}.
\begin{proposition}
	If $F_{\boldsymbol{m}}^{f,\delta,\boldsymbol{g}}$ is coercive on $\boldsymbol{\cB}^T_{\boldsymbol{m}}$, then there exists a coupled $\boldsymbol{g}$-weighted $(f,\delta)$-balanced $\boldsymbol{\varphi}\in\boldsymbol{\cB}^T_{\boldsymbol{m}}$, minimizing $F_{\boldsymbol{m}}^{f,\delta,\boldsymbol{g}}$.
\end{proposition}

The next result is the coupled version of Proposition \ref{prop:F-coerc=delta-m}.

\begin{proposition}
\label{prop:coupled-F-coercive=delta-m}
	$F_{\boldsymbol{m}}^{f,\delta,\boldsymbol{g}}$ is coercive on $\boldsymbol{\cB}^T_{\boldsymbol{m}}$ if and only if $\delta\in(0,\delta^{\boldsymbol{g}}_{\boldsymbol{m}}(L_1,\cdots,L_k))$.
\end{proposition}

\subsection{$\theta$-twisted coupled $\boldsymbol{g}$-balanced metric}
Recall that we have already chosen a $T$-invariant smooth form $\theta\in(c_1(X)-\sum_{i=1}^kc_1(L_i))$ at the begining. Consider
\begin{equation*}
	F^{\theta,\boldsymbol{g}}_{\boldsymbol{m}}(\boldsymbol{\varphi}):=-\log\int_Xe^{-\sum_{i=1}^k\varphi_i}d\mu-\sum_{i=1}^kE^{g_i}_{\omega_i,m_i}(\varphi_i),\ \boldsymbol{\varphi}=(\varphi_1,...,\varphi_k)\in\boldsymbol{\cB}^T_{
	\boldsymbol{m}}.
\end{equation*}

The next definition gives a natural quantization of the $\theta$-twisted coupled K\"ahler--Ricci $\boldsymbol{g}$-soliton (see also \cite{Tak19}).
\begin{definition}
	Any critical point of $F^{\alpha,\boldsymbol{g}}_{\boldsymbol{m}}$ is called $\theta$-twisted coupled $\boldsymbol{g}$-balanced.
\end{definition}

The next result, 
as a generalization of Theorem \ref{thm:finite-dim-YTD}, shows that the coupled $\delta^{\boldsymbol{g}}_{\boldsymbol{m}}$-invariant characterizes the existence of $\theta$-twisted coupled $\boldsymbol{g}$-balanced metrics.
\begin{theorem}
\label{thm:finite-coupled-dim-g-weighted-YTD}
	The following statements hold:	\begin{enumerate}
		\item If $\delta^{\boldsymbol{g}}_{\boldsymbol{m}}(L_1,...,L_k)>1$, then there exists an $\theta$-twisted coupled $\boldsymbol{g}$-balanced $\boldsymbol{\varphi}$ in $\boldsymbol{\cB}^T_{
	\boldsymbol{m}}$.
		\item Assume $\theta\geq0$. If there exists an $\theta$-twisted coupled $\boldsymbol{g}$-balanced $\boldsymbol{\varphi}$ (resp. a unique $\theta$-twisted coupled $\boldsymbol{g}$-balanced $\boldsymbol{\varphi}$ up to translation) in $\boldsymbol{\cB}^T_{
	\boldsymbol{m}}$, then $\delta^{\boldsymbol{g}}_{\boldsymbol{m}}(L_1,...,L_k)\geq1$ (resp. $\delta^{\boldsymbol{g}}_{\boldsymbol{m}}(L_1,...,L_k)>1$).
	\end{enumerate}
\end{theorem}

\subsection{Computing coupled $\delta_m$ using $T_\CC$-invariant divisors}
The proof of Theorem \ref{thm:delta-T-m} also works seamlessly for the coupled case. So we record the following result, without giving the proof.

\begin{theorem}
\label{thm:coupled-delta-T-m}
Let $X$ be K\"ahler manifold, polarized by a $k$-tuple $(L_1,\cdots,L_k)$ of ample $\QQ$-line bundles, together with a $T_\CC$-action. Put
$
s(L_1,\cdots,L_k):=\sup\big\{s\in\RR\,:\,-K_X-s\sum_{i=1}^kL_i\text{ is nef}\big\}.
$
Let $\boldsymbol{m}=(m_1,...,m_k)$ be a $k$-tuple of positive integers such that each $m_iL_i$ is very ample.  Then,
$$
\min\bigg\{s(L_1,\cdots,L_k),\delta^{T_\CC}_{\boldsymbol{m}}(L_1,\cdots,L_k)\bigg\}\leq\delta_{\boldsymbol{m}}(L_1,\cdots,L_k)\leq\delta^{T_\CC}_{\boldsymbol{m}}(L_1,\cdots,L_k).
$$
\end{theorem}

Now we apply the above result to the toric Fano case. Assume that $X$ is toric Fano and that
$
-K_X=\sum_{i}^kL_i.
$
Using the toric setup in  \S\ref{sec:delta-T}, we write
$
L_i=\sum_\rho a^i_\rho D_\rho,\ 1\leq i\leq k.
$
Up to linear equivalence, we may arrange that
$
\sum_{i=1}^ka^i_{\rho}=1,\ \forall\rho.
$
Each moment polytope $P_i$ is given by
$
P_i:=\{u\in M_\RR\ |\ \langle u,v_\rho\rangle+a^i_\rho\ge0,\ \forall\rho\}.
$
Note that all the weight spaces of $H^0(X,m_iL_i)$ are one-dimensional. Then by definition (recall \eqref{eq:def-m-g-basis-divisor}) there is only one
$(m_i,1)$-basis divisor of each $L_i$.
Thus,
$
\delta^{T_\CC}_{\boldsymbol{m}}(L_1,\cdots,L_k)=1/\max_\rho\big
\{{\langle \sum_{i=1}^kb_{m_i}(P_i),v_\rho\rangle+1}\big\},
$
where
$b_{m_i}(P_i)$
denotes the $m_i$-th quantized barycenter of $P_i$.
Hence,
$
\delta^{T_\CC}_{\boldsymbol{m}}(L_1,\cdots,L_k)\leq s(L_1,\cdots,L_k)=1.
$
So by Theorem \ref{thm:coupled-delta-T-m},
we obtain:

\begin{corollary}
$
\delta_{\boldsymbol{m}}(L_1,\cdots,L_k)=\delta^{T_\CC}_{\boldsymbol{m}}(L_1,\cdots,L_k)=
1/\max_\rho\big\{{\langle \sum_{i=1}^kb_{m_i}(P_i),v_\rho\rangle+1}\big\}.
$
\end{corollary}

Thus, sending $\boldsymbol{m}\rightarrow(\infty,...,\infty)$, 
$
\delta(L_1,\cdots,L_k):=\lim\delta_{\boldsymbol{m}}(L_1,\cdots,L_k)=1/\max_\rho\big\{{\langle \sum_{i=1}^kb(P_i),v_\rho\rangle+1}\big\},
$
where
$b(P_i)$
denotes the barycenter of $P_i$. In particular, one always has
$
\delta(L_1,\cdots,L_k)\leq1,
$
with equality if and only if
$
\sum_{i=1}^kb(P_i)=0.
$
Thus, using
a result of Hultgren
\cite[Theorem 2]{Hul19}
we obtain the following.
\begin{corollary}
A toric Fano manifold $X$ admits a coupled K\"ahler--Einstein tuple for $(L_1,\cdots,L_k)$ if and only if
$
\delta(L_1,\cdots,L_k)=1.
$
\end{corollary}

\section{Extension to klt currents}
\label{sec:current}
We use the setup of \cite{BBJ18}.
Let $(X,\theta,L)$ be a triple satisfying:
\begin{enumerate}
    \item $X$ is an $n$-dimensional projective manifold;
    
    \item $\theta$ is a quasi-positive $(1,1)$-current, i.e., the sum of a positive current and a smooth form;
    
    \item $\theta$ is klt, i.e., locally writing $\theta=\ddc\psi$, then $e^{-\psi}\in L^p_{\mathrm{loc}}$ for $p\in[1,1+\epsilon)$ for some $\epsilon>0$.
    
    \item $L$ is an ample $\QQ$-line bundle such that $c_1(L)=c_1(X)-[\theta]$.
\end{enumerate}
Let $\pi:Y\rightarrow X$ be a birational morphism. For any prime divisor $F\subset Y$ over $X$ let
$$
\ord_F(\theta)
$$
be the Lelong number of $\pis\theta$ at a very generic point of $F$ (see \cite{FJ05,BFJ08}). By \cite[Lemma 3.3]{BBJ18}, $\theta$ being klt is equivalent to
$
\inf_F\frac{A_X(F)}{\ord_F(\theta)}>1.
$
Set
\begin{equation}
\lb{Atheta}
    A_\theta(F):=A_X(F)-\ord_F(\theta).
\end{equation}
For any effective $\RR$-divisor $D$ on $X$, put
\begin{equation}
    \lct_\theta(X,D):=\inf_F\frac{A_\theta(F)}{\ord_F(D)}.
\end{equation}
Then \cite{BFJ08},
\begin{equation*}
    \lct_\theta(X,D)=\sup\{\lambda>0:\cJ(\theta+\lambda [D])=\cO_X\},
\end{equation*}
where $\cJ(\theta+\lambda [D])$ denotes the multiplier ideal sheaf associated to the current $\theta+\lambda[D]$ (and $[D]$ denotes the current of integration along $D$).

\begin{definition}
    The $\delta_m$-invariant of $(X,\theta,L)$ is 
    $
    \delta_m(L;\theta):=\inf\{\lct_\theta(X,D):D\text{ $m$-basis divisor of }L\}.
    $
\end{definition}
Equivalenty, one has
$
\delta_m(L;\theta)=\inf_F\frac{A_\theta(F)}{S_m(F)}.
$
In what follows we do not claim nor do we need to know whether the infimum is attained.

Next, we turn to analytic $\delta$-invariants. As before, fix a positively curved smooth Hermitian metric $h$ on $L$ and denote by $\omega$ 
its curvature form, with $[\o]=c_1(L)$. Let $f_\theta$ be a function 
satisfying
\begin{equation*}
    \Ric(\omega)=\omega+\theta+\ddc f_\theta, \quad
\int_Xe^{f_\theta}\omega^n=\int_X\omega^n=V.
\end{equation*}
More precisely, write $\theta=\theta_0+\ddc\psi$, where $\theta_0\in[\theta]$ is a smooth representative. Let $f_{\theta_0}\in C^\infty(X,\RR)$ satisfy $\Ric(\omega)=\omega+\theta_0+\ddc f_{\theta_0}$. Then (up to a constant) $f_\theta=f_{\theta_0}-\psi$.
Now define a probability measure on $X$,
\begin{equation*}
    d\mu_\theta:=\frac{e^{f_\theta}\omega^n}{V}.
\end{equation*}

\begin{definition}
    The analytic $\delta$-invariant of $(X,\theta,L)$ is defined by
    $$
    \delta^A(L;\theta):=\sup\bigg\{\delta>0:\sup_{\varphi\in\cH_\o}\int_Xe^{-\delta(\varphi-E(\varphi))}d\mu_\theta<\infty\bigg\}.
    $$
    The analytic $\delta_m$-invariant of $(X,\theta,L)$ is defined by
    $$
    \delta^A_m(L;\theta):=\sup\bigg\{\delta>0:\sup_{\varphi\in\cB_m}\int_Xe^{-\delta(\varphi-E_m(\varphi))}d\mu_\theta<\infty\bigg\}.
    $$
\end{definition}

\begin{theorem}
\lb{deltamthetaThm}
$\delta_m(L;\theta)=\delta^A_m(L;\theta)$.
\end{theorem}

\begin{proof}
    The proof parallels that of Theorem \ref{thm:delta-A-m=delta-m} given
in \S\ref{sec:analy-delta}. We point out the main differences for the reader's convenience. 

The proof of Theorem \ref{thm:delta-A-m=delta-m}
amounts to Corollaries 
\ref{cor:delta-m<=delta-m-A} and \ref{lowerbndcor}.
The analogue of the former, namely,
$\delta_m(L;\theta)\le\delta^A_m(L;\theta)$, 
is proven in much the same way as Corollary
\ref{cor:delta-m<=delta-m-A} with two key differences.
First, one needs to replace the volume form $\omega^n$ by the measure $d\mu_\theta$, and $A_X(F)$ by $A_\theta(F)$. Second, to apply Demailly--Koll\'ar's theorem \cite{DK01} as in the proof of Proposition \ref{prop:deltam=int<C}, one also needs to invoke the openness \cite{Bern15b,GZ15}.
The analogue of Corollary \ref{lowerbndcor}, namely,
 $\delta_m(L;\theta)\ge\delta^A_m(L;\theta)$ requires
an extension of Lemma \ref{lem:I-lambda>eta} to the setting of
non-zero Lelong number, given by Lemma \ref{lem:I>eta-theta} below.
\end{proof}

\begin{lemma}
    \label{lem:I>eta-theta}
    Let $F\subset Y\xrightarrow{\pi}X$ be a prime divisor over $X$. Assume that $\ord_F\theta>0$. Then for any $\delta>0$, $\varepsilon\in(0,\ord_F\theta)$, and any basis $\{s_i\}$ of $H^0(X,mL)$, there exists $C_\varepsilon>0$ such that  for any parameter $t\geq0$,
    $$
    \int_X\frac{e^{t(A_\theta(F)+\varepsilon)}}{\bigg(\sum_{i=1}^{d_m}e^{t\,\ord_F(s_i)}|s_i|^2_{h^m}\bigg)^{\frac{\delta}{m}}}d\mu_\theta\geq C_{\varepsilon}>0.
    $$
\end{lemma}

\begin{proof}
    We follow the proof of Lemma \ref{lem:I-lambda>eta}, by using a local calculation around a generic point of $F$, the only difference being that  in the presence of the current $\theta$, one should take into account the contribution coming from Lelong numbers. 
    
    More precisely, write $\theta=\theta_0+\ddc\psi$, where $\theta_0\in[\theta]$ is a smooth representative. Then it amounts to bounding the integral
    $$
    \int_X\frac{e^{t(A_\theta(F)+\varepsilon)}e^{-\psi}}{\bigg(\sum_{i=1}^{d_m}e^{t\,\ord_F(s_i)}|s_i|^2_{h^m}\bigg)^{\frac{\delta}{m}}}\omega^n
    $$
    from below for any $t\geq0$. We pull back everything to $Y$ and work in a polydisc $\DD$ around a very generic point of $F$, as in the proof of Lemma \ref{lem:I-lambda>eta}. By definition of $\ord_F\theta$, we may further assume that
    $$
    \pis\psi\leq (\ord_F(\theta)-\varepsilon)\log|z_1|^2+C_\varepsilon\text{ on }\DD 
    $$
    for some $C_\varepsilon>0$.
    Recalling \er{Atheta}, it suffices to bound
    \begin{equation*}
        \begin{aligned}
         J(t)&:=\sqrt{-1}\int_{|z_1|\leq1}\frac{e^{t(A_\theta(F)+\varepsilon)}|z_1|^{2A_{\theta}(F)+2\varepsilon-2}}{\bigg(\sum_{i=1}^{d_m}e^{t\,\ord_F(s_i)}|z_1|^{2\ord_F(s_i)}\bigg)^{\frac{\delta}{m}}}dz_1\wedge d\bar{z_1}\\
    &\geq\sqrt{-1}\int_{|w|\leq1}\frac{|w|^{2A_{\theta}(F)+2\varepsilon-2}}{\bigg(\sum_{i=1}^{d_m}|w|^{2\ord_F(s_i)}\bigg)^{\frac{\delta}{m}}}dw\wedge d\bar{w},\\
        \end{aligned}
    \end{equation*}
    which is a positive quantity only depending on $\delta$, $m$, $A_\theta(F)$, $\varepsilon$ and $\{\ord_F(s_i)\}$. 
\end{proof}

Now we turn to balanced metrics. For any $\delta>0$, put
\begin{equation*}
    F^{\theta,\delta}_m(\varphi):=-\frac{1}{\delta}\log\int_Xe^{-\delta\varphi}d\mu_\theta-E_m(\varphi),\ \varphi\in\cB_m.
\end{equation*}

\begin{definition}
    A critical point of $F^{\theta,\delta}_m$ is called $(\theta,\delta)$-balanced of level $m$. 
\end{definition}

One can also define coercivity for $F^{\theta,\delta}_m$ as in Definition \ref{def:coerc-F}. And as in Proposition \ref{prop:F-m-coercive=>balanced}, $F^{\theta,\delta}_m$ being coercive implies the existence of $(\theta,\delta)$-balanced metrics of level $m$.

\begin{proposition}
\label{prop:F-theta-coercive=delta-m-theta}
$F^{\theta,\delta}_m$ is coercive on $\cB_m$ if and only if $0<\delta<\delta_m(L;\theta)$.
\end{proposition}

\begin{proof}
After replacing the volume form $\omega^n$ by the measure $d\mu_\theta$ and $\alpha_m(L)$ by 
$$
\alpha_m(L;\theta):=\sup\bigg\{\lambda>0:\sup_{\varphi\in \cB_m}\int_Xe^{-\lambda(\varphi-\sup\varphi)}d\mu_\theta<\infty\bigg\},
$$
the proof goes through 
following the one for Proposition \ref{prop:F-coerc=delta-m}. 
\end{proof}

When $\delta=1$, we put
$
    F^\theta_m:=F^{\theta,1}_m
$
for simplicity and any critical point of $F^\theta_m$ is called \emph{$\theta$-balanced of level $m$}.
Then using Proposition \ref{prop:F-theta-coercive=delta-m-theta}, Berndtsson convexity \cite{Bern15} and Darvas--Rubinstein principle \cite{DR17}, we get the following quantized version of \cite[Theorem A]{BBJ18}.

\begin{theorem}
\label{thm:finite-dim-YTD-current}
{\rm (Algebraic characterization of $\theta$-balanced metrics)}
One has
\begin{enumerate}[(i)]
    \item if $\delta_m(L;\theta)>1$ there exists a
    $\theta$-balanced metric of level $m$;
    
    \item suppose $\theta$ is semipositive. If there exists a $\theta$-balanced metric of level $m$ then $\delta_m(L;\theta)\geq1$.
    If such a metric is unique then $\delta_m(L;\theta)>1$.
\end{enumerate}
\end{theorem}

Finally we remark that, with the help of \cite{DLR19,Bern15b,GZ15}, all the results in \S\ref{sec:limit-behavior} (except Proposition \ref{prop:tKE-appro-smoothly-by-balanced}) and \S\ref{sec:delta-T} can be established in the current setting.
One can also extend the above discussions to soliton type metrics, as in \S\ref{sec:soliton}, and even to coupled soliton metrics, as in Appendix \ref{sec:coupled-soliton}. But we shall not pursue such generality here. 

\medskip
\noi
\textbf{Acknowledgments.}
We thank A. Lahdili and F. Wang for  helpful discussions.
The research of Y.A.R. was supported by
		NSF grants DMS-1515703,1906370 and the Rosi \& Max Varon Visiting
		Professorship at the Weizmann Institute of Science in Fall 2019 and Spring 2020. G.T. was supported by NSFC grants 11331001 and 11890661.
		K.Z. was supported by the CSC award 201706010020 and China post-doctoral grant BX20190014. K.Z. was a Visiting Scholar at the University of Maryland in 2017--2018 when part of this work was initiated and is grateful to T. Darvas for many inspiring conversations back then.

\bibliography{ref.bib}
\bibliographystyle{abbrv}

\end{document}